\documentclass[final]{siamart251216}
\usepackage[backend=biber,sorting=none]{biblatex}
\usepackage{amsmath}
\usepackage{amssymb}
\usepackage{makecell}
\usepackage{hyperref}
\usepackage{cleveref}
\usepackage{graphicx}
\usepackage{amsfonts}
\usepackage{float}
\usepackage{mathtools}
\usepackage{csquotes}
\usepackage{microtype}
\allowdisplaybreaks

\usepackage[textsize=tiny]{todonotes}
\AtBeginBibliography{\emergencystretch 3em}

\newtheorem{thm}{Theorem}

\newtheorem{lem}{Lemma}
\newtheorem{cor}{Corollary}
\newtheorem{prop}{Proposition}

\newtheorem{defn}{Definition}

\newtheorem{rem}{Remark}
\newtheorem{Example}{Example}
\DeclareMathOperator*{\dx}{\, \text{d}x}

\DeclareMathOperator*{\SPD}{\mathcal{S}_{++}}
\DeclareMathOperator*{\argmin}{arg\,min}
\DeclareMathOperator*{\argmax}{arg\,max}
\DeclareMathOperator*{\tr}{\mathrm{tr}}
\DeclareMathOperator*{\GC}{\mathop{\Gamma\text{-}\lim}}
\title{Existence of Minimizers for the Gaussian Mumford--Shah Model via Covariance Floors on the Value Distribution\thanks{\textbf{Funding:} Fabian Conrad was funded by the Deutsche Forschungsgemeinschaft (DFG, German Research Foundation) -- 333849990/GRK2379 (IRTG Modern Inverse Problems).}}
\author{Fabian Conrad\thanks{IGPM, RWTH Aachen University, Germany (\email{conrad@ssd.rwth-aachen.de}).}}
\date{}
\begin{document}
\maketitle
\begin{abstract}
The Gaussian Mumford--Shah model segments an image into regions by minimizing a continuous multivariate Gaussian negative log-likelihood term plus a perimeter penalty. It is used for scalar- and vector-valued images, in particular hyperspectral images, but the existence of a minimizer in general is not guaranteed. The energy is unbounded below whenever the covariance of a single region degenerates, regardless of the perimeter of the regions. A quantitative covariance floor on the covariances of all regions, or equivalently uniform Hölder continuity with exponent at most one of all one-dimensional marginals of the histogram, guarantees a lower bound on the energy and the existence of a minimizer. Our approach therefore shows that existence follows from a  property of the histogram of the image, and we furthermore argue that scale-independent covariance floors cannot exist. Admitting a quantitative covariance floor is generic: it holds almost surely for all images with suitable noise. As applications of our theory, we show the convergence of regularized models used in practice and of numerical methods.
\end{abstract}
\begin{keywords}
Mumford--Shah functional, variational image segmentation, Gaussian models, covariance degeneracy, existence of minimizers
\end{keywords}

\begin{MSCcodes}
49J45, 49Q20, 68U10, 94A08, 62H35
\end{MSCcodes}

\section{Introduction}
Region competition with multivariate Gaussian statistics \cite{ST1}, which we will refer to as the Gaussian Mumford--Shah model, and versions of it are used in the context of image segmentation \cite{RoussonDeriche,Cohrs_2022,Conrad_Berkels_2025_Hyperspectral, seaice}, providing, for example, competitive results on benchmark hyperspectral datasets. The multivariate case has full covariance matrices, which capture the high correlation between bands in hyperspectral imaging. Let $g \in L^{\infty}((0,1)^2,\mathbb{R}^N)$ be a bounded measurable image. For a fixed number of segments $k \in \mathbb{N}$, 
we consider a partition $\mathcal{O}=(O_1, \dots,O_k)$ of $(0,1)^2$ with $O_l \subset (0,1)^2$ having finite perimeter in $(0,1)^2$, denoted by $\mathrm{Per}(O_l,(0,1)^2)$, such that
\[
\bigcup_{l \in \{1, \dots, k \} } O_l=(0,1)^2 \quad \wedge \quad \forall l \neq l': \, O_l \cap O_{l'}=\emptyset.
\]
We denote the set of all such partitions by $P_k$. Let $\SPD(N)$ denote the set of symmetric positive definite matrices in $\mathbb{R}^{N \times N}$. 
For any $S \in \SPD(N)$, we define the Mahalanobis norm of $z \in \mathbb{R}^N$ by
\[
\vert z \vert_{S}^2:=z^TS^{-1}z.
\]
Then, we define the  \textbf{Gaussian Mumford--Shah model (GMS)} by
\begin{equation}\label{Def: GMS}
\begin{split}
E:{}&P_k \times (\mathbb{R}^N)^k \times \SPD(N)^k  \to \mathbb{R}, \\
&(\mathcal{O},\mu,\Sigma) \mapsto \sum_{l=1}^k \bigg( \int_{O_l} f(g(x) \mid \mu_l, \Sigma_l) \dx + \lambda \mathrm{Per}(O_l,(0,1)^2)\bigg),
\end{split}
\end{equation}
where
\[
f(g(x) \mid \mu_l, \Sigma_l)=\frac{1}{2} \vert g(x)-\mu_l \vert_{\Sigma_l}^2 + \frac{1}{2} \ln \det \Sigma_l +\frac{N}{2} \ln 2 \pi
\]
is the indicator function given by the negative log-likelihood of a multivariate Gaussian $\mathcal{N}(\mu_l,\Sigma_l)$ and $\lambda>0$ is a fixed length parameter.  If we set $\Sigma=(\Sigma_1, \dots,\Sigma_k)=(I,\dots,I)$ to be the identity matrix $I$ in every component, we recover the classical piecewise constant Mumford--Shah model
\cite{MS} up to an additive constant, for which existence and even regularity theory is well-known \cite{Maggi}. However, the existence of a minimizer of the GMS \eqref{Def: GMS}, and by extension the region competition model \cite{ST1}, has not been addressed so far to our knowledge. When applying the
direct method of the calculus of variations, the main obstructions are the unboundedness from below, which prevents the extraction of a minimizing sequence converging to a finite infimum, the fact that $\SPD(N)$ is not closed and hence might not contain the limit of such a sequence, and the lack of lower semi-continuity. Following the arguments of \cite{Cohrs_2024}, if $g(O)$ is contained in some affine hyperplane $H \subset \mathbb{R}^N$ for some $O \subset (0,1)^2$
with positive measure $\vert O \vert >0$ and finite perimeter, we can find a sequence of statistical variables $(\mu^{(n)},\Sigma^{(n)})_{n \in \mathbb{N}} \subset (\mathbb{R}^N)^k \times \SPD(N)^k$ such that for the fixed partition $(O,(0,1)^2\setminus O,\emptyset,\dots,\emptyset) \in P_k$, which in particular has a fixed perimeter contribution,
the energy diverges to $-\infty$, i.e.
\[
\lim_{n \to \infty} E[(O,(0,1)^2\setminus O,\emptyset,\dots,\emptyset), \mu^{(n)},\Sigma^{(n)} ]=-\infty.
\]
Let us introduce the (continuous) \textbf{sample mean} and \textbf{sample covariance} for any $O \subset (0,1)^2$ with $\vert O \vert >0$ by
\begin{equation}
\begin{split}
m_O:&=\frac{1}{\vert O \vert} \int_O g(x)\dx \quad \text{and} \\
 \quad C_O:&=\frac{1}{\vert O \vert} \int_{O} \big( g(x)-m_O \big) \otimes \big( g(x)-m_O \big) \dx.
\end{split}
\end{equation}
For the case $\vert O \vert=0$, this definition is vacuous. We introduce the \textbf{covariance floor condition (CFC)} on $g \in L^{\infty}((0,1)^2,\mathbb{R}^N)$, which means there exist $C>0$ and
$p >0$ such that
\begin{equation}\label{Def: CFC-Intro}
y^TC_Oy \geq C\vert O \vert^p 
\end{equation}
for all $y \in S^{N-1}:=\{y \in \mathbb{R}^N \mid \vert y \vert=1 \}$ and $O \subset (0,1)^2$ with $\vert O \vert >0$. Equivalently, $C_O \succeq C \vert O \vert^p I$, where  $\succeq$ denotes the Loewner order on symmetric matrices, i.e. $A \succeq B$ means that $A-B$ is positive semi-definite for symmetric matrices $A,B \in \mathbb{R}^{N \times N}$ . We show that for $g \in L^{\infty}((0,1)^2,\mathbb{R}^N)$, this condition can only hold if $p \geq 2$. In particular, a uniform (scale-independent) covariance floor $C_O \succeq CI$ for all $O \subset (0,1)^2$ with $\vert O \vert >0$ cannot exist. The CFC \eqref{Def: CFC-Intro} is a property of the pushforward measure $g_{\sharp} \mathcal{L}^2$. Thus, it is intrinsic to the data. We also prove that the CFC is equivalent to a uniform Hölder bound on the cumulative distributions of the pushforward measure $g_{\sharp} \mathcal{L}^2$. In research and literature, this concept is related to the Frostman lemma \cite{mattila1995geometry}, occupation densities \cite{OccupationDensity} and Lévy concentration functions \cite{ConcentrationFunctions}, and we will refer to these at the respective points. Most notably, none of these three concepts requires uniformity in $y \in S^{N-1}$, which the CFC does. The discrete sample means and covariance matrices are studied at a fixed scale, whereas the continuum sample covariances need lower bounds at every scale. Non-degeneracy of the sample covariance in the discrete case depends on having at least $N+1$ affinely independent samples, while in the continuous case, the question is how fast $C_{O}$ may degenerate as $\vert O \vert \to 0$.

Moreover, we show that adding suitable noise to any image yields a
covariance floor almost surely. For $\sigma>0$ and a suitable fractional Gaussian field (restricted from $\mathbb{R}^2$ to $(0,1)^2$) $W$, the noisy image
\[
g + \sigma W
\]
satisfies the CFC almost surely and any image can be approximated by an image satisfying the CFC arbitrarily well in the $L^{\infty}((0,1)^2,\mathbb{R}^N)$ topology. Hence, the CFC is not vacuous and feasible.

We show that imposing the CFC on $g$ implies both the existence of an entropy expression as a lower bound for the GMS and the existence of a uniform lower bound for the Gaussian Mumford--Shah model. As the first application of the CFC (and the existence of lower bounds for the GMS), we show 
the convergence of minimizers, including convergence rates for the energies, of the \textbf{perturbed Gaussian Mumford--Shah model (pGMS)}
\begin{equation}\label{Def: pGMS}
E^{(\epsilon)}[\mathcal{O},\mu,\Sigma]:=E[\mathcal{O},\mu,\Sigma]+\frac{\epsilon}{2} \sum_{l=1}^k \vert O_l \vert \mathrm{tr} \Sigma_l^{-1}
\end{equation}
to minimizers of the GMS by using $\Gamma$-convergence of reduced, partition-only based versions of $E^{(\epsilon)}$ and $E^{(0)}:=E$ as $\epsilon \to 0$. The pGMS is the variational formulation of regularizing the sample covariance matrices by $C_{O_l} \gets C_{O_l} + \epsilon I \succeq \epsilon I$ ($l=1, \dots, k$), a common way to enforce positive definiteness in practice. A closely related regularization, clipping the eigenvalues of $C_{O_l}$ from below, corresponds to the constraint $\Sigma_l \succeq \epsilon I$ and was used in \cite{Cohrs_2022} for a model with a non-squared Mahalanobis data term. For that model, \cite{Cohrs_2024} showed existence for all $g \in L^{\infty}((0,1)^2,\mathbb{R}^N)$ as long as $\epsilon>0$ and $\Gamma$-convergence as $\epsilon \to 0$, but not necessarily convergence of minimizers. Without the CFC, the GMS itself can be unbounded from below, and the regularization does not remove this singularity in the limit $\epsilon \to 0$. For images satisfying the CFC, this singularity does not occur. Consequently, the GMS \eqref{Def: GMS} admits a minimizer for every $g$ satisfying the CFC \eqref{Def: CFC-Intro}. As the second application, we show that after replacing the labels $\mathcal{O}$ in the GMS \eqref{Def: GMS} with a convex relaxation $u$,
and alternately updating the labels $u$ and the statistical variables $(\mu,\Sigma)$, we have convergence of a subsequence. The second application highlights that commonly applied numerical approaches to these models also converge, at least along a subsequence. In both of these applications, a lower bound was the biggest missing ingredient so far.

This paper starts by discussing the theoretical aspects of the CFC. In particular, we show that satisfying the CFC is equivalent to the marginal distributions of $g$ being Hölder continuous uniformly in $y \in S^{N-1}$ and that $p \geq 2$ is necessary. 
It is then argued that the pushforward measure of $g$, the continuous version of the histogram, is the correct invariant for the existence of covariance floors. Using techniques from Fourier analysis, we prove that a noisy image will almost surely satisfy the CFC. This section is independent of the Gaussian Mumford--Shah context. In the following section, we use information theory to show that the GMS admits a lower bound for images satisfying the CFC. With this lower bound at hand, we turn to our applications, which cover the existence of minimizers of the GMS and the convergence of the alternating minimization. The final section contains some numerical experiments illustrating our main theorems.

\section{Covariance floors for images}

\begin{defn}[Covariance floor condition]\label{Def: CFC}
Let $g \in L^{\infty}((0,1)^2,\mathbb{R}^N)$. We say $g$ satisfies the
\textbf{covariance floor condition (CFC)} with exponent $p > 0$ if there
exists $C>0$ such that
\[
C_O \succeq C \vert O \vert^{p} I
\]
for all measurable $O \subset (0,1)^2$ with $\vert O \vert > 0$.
\end{defn}

We also introduce the pushforward measure of $g$, which we denote by $\nu$ without reference to $g$ since $g$ is fixed,
\[
\nu(A):=g_{\sharp}\big(\mathcal{L}^2\big\vert_{(0,1)^2}\big)(A)=\vert \{ x \in (0,1)^2 \mid g(x) \in A \} \vert
\]
for any Borel set $A \subset \mathbb{R}^N$. Similarly, we define for any $y \in S^{N-1}$ the projected pushforward measure 
\[
\nu_y(B):=(y^Tg)_{\sharp}\big(\mathcal{L}^2\big\vert_{(0,1)^2}\big)(B)=\vert \{ x \in (0,1)^2 \mid y^Tg(x) \in B \} \vert
\]
for any Borel set $B \subset \mathbb{R}$. By the Cramér--Wold theorem \cite{ProbabilityAndMeasure}, the marginals $\nu_y$ completely determine $\nu$ and vice versa. 
Closely related is the following definition: 

\begin{defn}[Regular distribution condition]\label{Def: RDC}
Let $g \in L^{\infty}((0,1)^2,\mathbb{R}^N)$. We say $g$ satisfies the \textbf{regular distribution condition (RDC)} with exponent $p > 0$, if for all $y \in S^{N-1}$, the cumulative
distributions
\[
F_y:\mathbb{R}\to[0,1],\ t\mapsto \nu_y((-\infty,t])
\]
are $\frac{2}{p}$-Hölder continuous with uniform Hölder constant, i.e. there exists $C>0$ independent of $y \in S^{N-1}$ such that
\[
\sup_{y \in S^{N-1}} \Vert F_y \Vert_{C^{0,\frac{2}{p}}(\mathbb{R})} \leq C.
\]
\end{defn}

We will now state the main theorem of this section, which gives an equivalent characterization of the CFC.

\begin{thm}[Equivalence of conditions]\label{Thm: Equivalence}
Let $g \in L^{\infty}((0,1)^2,\mathbb{R}^N)$ and let $p>0$. Then the following are equivalent:
\begin{itemize}
    \item $g$ satisfies the CFC with exponent $p$.
    \item $g$ satisfies the RDC with exponent $p$.
\end{itemize}
Neither condition can be satisfied by any $g \in L^{\infty}((0,1)^2,\mathbb{R}^N)$ if $p<2$.
\end{thm}
\begin{rem}[Dependence on the histogram]\label{Rem: Histogram}
Theorem \ref{Thm: Equivalence} shows that whether $g$ admits a covariance floor depends only on its value distribution and not on its spatial arrangement. In the language of imaging sciences, the CFC only depends on the (continuous) histogram of the image $g$, which is the pushforward measure $\nu$. Furthermore, the requirement $p \geq 2$ shows us that on small scales $\vert O \vert \ll 1$, the continuous sample covariance matrix $C_O$ cannot shrink arbitrarily slowly.
In particular, demanding a universal lower bound for all scales, 
\[
C_O \succeq CI 
\]
for all $O \subset (0,1)^2$ with $\vert O \vert >0$, can never be achieved.
\end{rem}

Furthermore, the sample covariance matrix does not even have to shrink at all and hence pinching the covariance matrix becomes impossible:

Let $g \in L^{\infty}((0,1)^2,\mathbb{R}^N)$ have two Lebesgue points $x_1 \neq x_2$ such that $g(x_1)=v_1 \neq v_2=g(x_2)$. Let us set $O_r=B_r(x_1) \cup B_r(x_2)$ with
$B_r(x_i):=\{ x \in (0,1)^2 \mid \vert x-x_i \vert \leq r \}$ $(i=1,2)$. Let $r>0$ be so small that $B_r(x_1)$ and $B_r(x_2)$ are disjoint and contained in $(0,1)^2$. Then each ball satisfies $\vert B_r(x_1) \vert =\vert B_r(x_2) \vert = \frac{\vert O_r \vert}{2}$, and since $x_1,x_2$ are also Lebesgue points of $g \otimes g$ (as $g \in L^{\infty}((0,1)^2,\mathbb{R}^N)$), we obtain
\begin{align*}
&m_{O_r}= \frac{1}{2} \bigg( \frac{1}{\vert B_r(x_1) \vert }\int_{B_r(x_1)}g(x) \dx + \frac{1}{\vert B_r(x_2) \vert } \int_{B_r(x_2)} g(x) \dx \bigg)  \to \frac{v_1+v_2}{2} \quad \text{and} \\
&\frac{1}{\vert O_r \vert}\int_{O_r} g(x) \otimes g(x) \dx \to \frac{v_1 \otimes v_1+v_2 \otimes v_2}{2} 
\end{align*}
as $r \to 0$. Using $C_{O_r}=\frac{1}{\vert O_r \vert}\int_{O_r} g(x) \otimes g(x) \dx{} - m_{O_r} \otimes m_{O_r}$, we conclude
\[
\lim_{r \to 0} C_{O_r}= \frac{v_1 \otimes v_1+v_2 \otimes v_2}{2}-\frac{(v_1+v_2) \otimes (v_1+v_2)}{4}=\frac{(v_1-v_2) \otimes (v_1-v_2)}{4},
\]
for which $\frac{\vert v_1-v_2 \vert^2}{4}$ is a non-zero eigenvalue and hence no bound of the form $C_{O_r} \preceq C \vert O_r \vert^{p'}I$ can hold for any $p'>0$.

We will now prove Theorem \ref{Thm: Equivalence}:
\begin{proof}
($\implies$): Let the CFC hold with exponent $p > 0$, constant $C>0$ and let $y \in S^{N-1}$ be arbitrary. Let $B \subset \mathbb{R}$ be any Borel set and set $O=(y^Tg)^{-1}(B)$ with
$\vert O \vert >0$. By definition of the pushforward measure, we have $\nu_y(B)=\vert O \vert$. The CFC with constant $C>0$ and exponent $p > 0$ then implies
\[
C \nu_y(B)^p =C \vert O \vert^p \leq y^T C_O y= \frac{1}{\vert O \vert} \int_O (y^Tg(x)-y^Tm_O)^2 \dx \leq \frac{\mathrm{diam}(B)^2}{4}, 
\]
since $y^Tg$ takes values only in $B \subset [\inf B,\sup B]$ on $O$, and by Popoviciu's inequality \cite{BhatiaDavis2000} the variance of a real random variable with values in an interval of length $\ell$ is at most $\ell^2/4$, where $\sup B-\inf B=\mathrm{diam}(B)$ for $B \subset \mathbb{R}$. For $\vert O \vert=0$ or infinite diameter $\mathrm{diam}(B)=\infty$, the inequality becomes trivial. Re-arranging this 
inequality and taking the supremum yields
\[
\sup_{y \in S^{N-1}}\nu_y(B) \leq \frac{1}{(4C)^{\frac{1}{p}}} \mathrm{diam}(B)^{\frac{2}{p}}.
\]
Let $t_1 < t_2$ be two distinct real numbers. In particular, the inequality above holds for $B=[t_1,t_2]$ and hence we have for the 
cumulative distribution
\[
\sup_{y \in S^{N-1}} \vert F_y(t_1)-F_y(t_2) \vert \leq \sup_{y \in S^{N-1}} \nu_y((t_1,t_2]) \leq  \frac{1}{(4C)^{\frac{1}{p}}} \vert t_1-t_2 \vert^{\frac{2}{p}},
\]
which directly implies the RDC with exponent $p$ and constant $\frac{1}{(4C)^{\frac{1}{p}}}$.

($\impliedby$): Conversely, let the RDC hold with exponent $p > 0$, constant $C>0$ and let $y \in S^{N-1}$ be arbitrary. Let $O \subset (0,1)^2$
with $\vert O \vert >0$ be arbitrary. For any Borel set $B \subset \mathbb{R}$, we set $I_B:=[\inf B, \sup B]$ (possibly infinite). Since $F_y$ is continuous by the RDC, $\nu_y$ has no atoms, i.e., $\nu_y(\{t\})=0$ for all $t \in \mathbb{R}$, and hence $\nu_y(I_B)=\nu_y((\inf B,\sup B])=F_y(\sup B)-F_y(\inf B)$. Therefore,
\[
\nu_y(B) \leq \nu_y(I_B) \leq \vert F_y(\sup B) -F_y(\inf B) \vert \leq C \vert \sup B-\inf B \vert^{\frac{2}{p}}\leq C \mathrm{diam}(B)^{\frac{2}{p}},
\]
since $B \subset I_B$ and the RDC holds. Consequently, the restricted pushforward measure
\[
\nu_{y,O}(B):=\big\vert O \cap (y^Tg)^{-1}(B) \big\vert
\]
satisfies
\begin{equation}\label{Measure condition I}
\nu_{y,O}(B) \leq \nu_y(B) \leq C \mathrm{diam}(B)^{\frac{2}{p}}
\end{equation}
for any Borel set $B \subset \mathbb{R}$ and 
\begin{equation}\label{Measure condition II}
\nu_{y,O}(\mathbb{R})=\vert O \vert.
\end{equation}
We will now execute a bathtub principle style argument \cite{LiebLoss} to show that 
\[
\int_{\mathbb{R}} (t-\bar{t})^2\text{d}\nu_{y,O}(t)
\]
for $\bar{t}= \frac{1}{\nu_{y,O}(\mathbb{R})}\int_{\mathbb{R}}t \text{d}\nu_{y,O}(t)$ is bounded below by a positive constant independent of $y$. For $r=(\frac{\vert O \vert}{2C})^{\frac{p}{2}}>0$
and the interval $I=[\bar{t}-\frac{r}{2},\bar{t}+\frac{r}{2}]$, we have by 
\[
\nu_{y,O}(I) \leq C r^{\frac{2}{p}}=\frac{\vert O \vert}{2}.
\]
The second condition \eqref{Measure condition II} then gives us
\[
\nu_{y,O}(\mathbb{R} \setminus I)= \vert O \vert-\nu_{y,O}(I) \geq \frac{\vert O \vert }{2}.
\]
Since we have $(t-\bar{t})^2 \geq \frac{r^2}{4}$ for all $t \in \mathbb{R} \setminus I$, we have
\[
\int_{\mathbb{R}} (t-\bar{t})^2\text{d}\nu_{y,O}(t) \geq \int_{\mathbb{R} \setminus I } (t-\bar{t})^2 \text{d}\nu_{y,O}(t) \geq \frac{r^2}{4} \frac{\vert O \vert}{2}=\frac{1}{2^{p+3}C^p} \vert O \vert^{p+1}.
\]
We conclude the proof by noting that, by definition of the restricted pushforward measure, 
\begin{align*}
y^T C_O y&=\frac{1}{\vert O \vert} \int_O (y^Tg(x)-y^Tm_O)^2 \dx   \\
& = \frac{1}{\vert O \vert } \int_{\mathbb{R}} (t-\bar{t})^2\text{d}\nu_{y,O}(t) \geq \frac{1}{2^{p+3}C^p} \vert O \vert^{p},
\end{align*}
which by the Courant--Fischer--Weyl principle for eigenvalues \cite{JoharinadJost2023} implies the CFC with exponent $p$ and constant $\frac{1}{2^{p+3}C^p}$.

($p \geq 2$): Lastly, we show that $p \geq 2$ is required and we will show that the marginals are closely related to  Frostman measures \cite{mattila1995geometry} uniformly in $S^{N-1}$. Let $B \subset \mathbb{R}$ be 
any Borel set. For $\delta >0$, let $B \subset \bigcup_{i \in I}U_i$ be a countable covering of $B$ by sets with $\sup_{i \in I}\mathrm{diam}(U_i) < \delta$.
Assuming the CFC holds with exponent $p>0$ and constant $C>0$, we have by the previous arguments for any $y \in S^{N-1}$
\[
\nu_y(B) \leq \sum_{i \in I} \nu_y(U_i) \leq C \sum_{i \in I} \mathrm{diam}(U_i)^{\frac{2}{p}}.
\]
Taking the infimum over all such covers and sending $\delta \to 0$ yields
\[
\nu_y(B) \leq C \mathcal{H}^{\frac{2}{p}}(B)
\]
for the $\frac{2}{p}$-dimensional Hausdorff measure $\mathcal{H}^{\frac{2}{p}}$. In particular, we must have 
\[
1=\nu_y(\mathbb{R}) \leq C \mathcal{H}^{\frac{2}{p}}(\mathbb{R}),
\]
which is a contradiction for $p<2$, since $\mathcal{H}^{\alpha}(\mathbb{R})=0$ for $\alpha>1$.
\end{proof}

Provided $\nu_y$ comes from a density $h_y$, i.e. 
\[
\nu_y(B)=\int_{(0,1)^2} \mathbf{1}_B(y^Tg(x))\dx= \int_{B} h_y(t) \text{d}t
\]
holds for all Borel sets $B \subset \mathbb{R}$, we see that if the (marginal) densities $h_y$ satisfy $\sup_{y \in S^{N-1}} \Vert h_y \Vert_{L^{\frac{p}{p-2}}} \leq C < \infty$, with the convention $\frac{p}{0}=\infty$, we immediately get
\begin{equation}\label{Estimate: Density gives RDC}
\nu_y((t_1,t_2])=\int_{t_1}^{t_2} h_y(t) \text{d}t \leq \sup_{y \in S^{N-1}} \Vert h_y \Vert_{L^{\frac{p}{p-2}}} \vert t_1-t_2 \vert^{\frac{2}{p}} \leq C \vert t_1-t_2 \vert^{\frac{2}{p}}
\end{equation}
by Hölder's inequality. Taking the supremum over $y \in S^{N-1}$ yields the RDC with exponent $p \geq 2$. If $g$ is smooth with $\vert \{ \nabla y^Tg=0 \} \vert =0$, an application of the co-area formula \cite{EvansMT} reveals that for any $y \in S^{N-1}$ and Borel set $B \subset \mathbb{R}$ we have
\[
\nu_y(B)=\int_{(0,1)^2} \mathbf{1}_B(y^Tg(x))\dx= \int_{B} \underbrace{\bigg( \int_{(y^Tg)^{-1}(t)} \frac{1}{\vert \nabla (y^Tg)(x) \vert} \text{d}\mathcal{H}^1(x) \bigg)}_{=h_y(t)} \text{d}t.
\]
While spatial regularity is not needed for a map $g \in L^{\infty}((0,1)^2,\mathbb{R}^N)$ to satisfy the RDC, one can impose regularity conditions on $g$ which imply
$\sup_{y \in S^{N-1}} \Vert h_y \Vert_{L^{\frac{p}{p-2}}} \leq C < \infty$.
\begin{Example} [Affine maps in $\mathbb{R}^2$]
As an immediate application, we see that affine maps to $\mathbb{R}^2$ 
\[
g(x)=Ax+b
\]
with $A \in \mathrm{GL}(2)$ and $b \in \mathbb{R}^2$ satisfy the CFC. In this case, we see that 
\[
\vert h_y(t) \vert \leq \frac{1}{\vert A^Ty \vert} \sup_{s \in \mathbb{R}} \mathcal{H}^1\big( \{x \in (0,1)^2 \mid y^TAx+y^Tb=s \} \big) \leq \frac{\sqrt{2}}{\sigma_{\min}(A)},
\]
where $\sigma_{\min}(A)$ is the smallest singular value of $A$. This implies $\sup_{y \in S^{N-1}} \Vert h_y \Vert_{L^{\infty}} \leq \frac{\sqrt{2}}{\sigma_{\min}(A)}$ and by the same argument as above \eqref{Estimate: Density gives RDC}, 
affine maps to $\mathbb{R}^2$ satisfy the RDC with exponent $p=2$.
\end{Example}

Our analysis so far motivates the following definition:
\begin{defn}[Non-degenerate images]\label{Def: Space of ND images}
For $p \geq 2$, we define the space of \textbf{non-degenerate images} with exponent $p$ as 
\[
\mathcal{A}(p):=\big\{ g \in L^{\infty}((0,1)^2,\mathbb{R}^N) \mid g \text{ satisfies the CFC with exponent }p \big\}.
\]
\end{defn}

We record some properties and invariances of $\mathcal{A}(p)$:
\begin{lem}[Structure of $\mathcal{A}(p)$]
Let $g \in \mathcal{A}(p)$ for $p \geq 2$. Then we have
\begin{itemize}
    \item $Ag+b \in \mathcal{A}(p)$ for $A \in \mathrm{GL}(N), b \in \mathbb{R}^N$,
    \item $g \circ \tau \in \mathcal{A}(p)$ for
    \[
    \tau \in T:=\{t:(0,1)^2 \to (0,1)^2 \mid t \text{ preserves }\mathcal{L}^2 \},
    \]
    \item and $\mathcal{A}(p) \subset \mathcal{A}(p')$ for all $2 \leq p \leq p'$.
\end{itemize}
\end{lem}
\begin{proof}
Let $g \in \mathcal{A}(p)$ be arbitrary. 

For any $b \in \mathbb{R}^N$, we note that $g$ and $g+b$ have the same covariance matrix and $g$ satisfies the CFC.
Let $A \in \mathrm{GL}(N)$ and $O \subset (0,1)^2$ be arbitrary with $\vert O \vert >0$. For any $y \in S^{N-1}$, the covariance matrix $\tilde{C}_O$ of $Ag$ satisfies 
\[
y^T \tilde{C}_O y= (A^Ty)^T C_O (A^Ty) \geq C \vert O \vert^p \vert A^Ty \vert^2 \geq C \vert O \vert^p \sigma_{\min}(A)^2,
\]
where again $\sigma_{\min}(A)$ is the smallest singular value of $A$ and $C>0$ is the constant coming from the CFC. Hence $Ag$ satisfies the CFC with
the same exponent $p$ and constant $C\sigma_{\min}(A)^2>0$. Combining these facts, we have $Ag+b \in \mathcal{A}(p)$.  

For the invariance under measure-preserving maps, it suffices to note that for all $y \in S^{N-1}$, $y^Tg$ and $(y^Tg) \circ \tau$ have the same cumulative distribution and therefore
both satisfy the RDC with the same constant and exponent $p$.

Since $0 \leq \vert O \vert \leq 1$, we have $ \vert O \vert^{p} \geq \vert O \vert^{p'}$ and hence 
\[
C_O \succeq C \vert O \vert^{p}I \succeq C \vert O \vert^{p'} I
\]
for all $O \subset (0,1)^2$ with $\vert O \vert >0$ and hence $g \in \mathcal{A}(p')$.
\end{proof}
We would like to note that $\mathcal{A}(p)$ is not a convex cone, since
\[
0=\frac{1}{2}g+ \frac{1}{2}(-g) \notin \mathcal{A}(p)
\]
for any $p \geq 2$. It is therefore also not a vector space.

In the imaging sciences, in particular for motif extraction \cite{Motif}, we often have repeated patterns. To check the CFC on the whole image, it is sufficient to check it on every tile:

\begin{lem}[Tiling lemma]\label{Tiling lemma}
Let $g \in L^{\infty}((0,1)^2,\mathbb{R}^N)$ and
\[
(0,1)^2=\bigcup_{h,w=1}^{M}\underbrace{(\frac{(h-1)}{M},\frac{h}{M}) \times (\frac{(w-1)}{M},\frac{w}{M})}_{:=Q_{h,w}}
\]
be a partition (up to null sets) of $(0,1)^2$ into squares. Let $\tau_{h,w}:((0,1)^2,\mathcal{L}^2)\to (Q_{h,w},M^2\mathcal{L}^2)$ be any measure-preserving map.
If $g\vert_{Q_{h,w}}  \circ \tau_{h,w} \in \mathcal{A}(p_{h,w})$ for all $1 \leq h,w \leq M$, then $g \in \mathcal{A}(p)$ for $p=\max_{h,w}p_{h,w}$.
\end{lem}
\begin{figure}[H]
    \centering
    \includegraphics[scale=0.75]{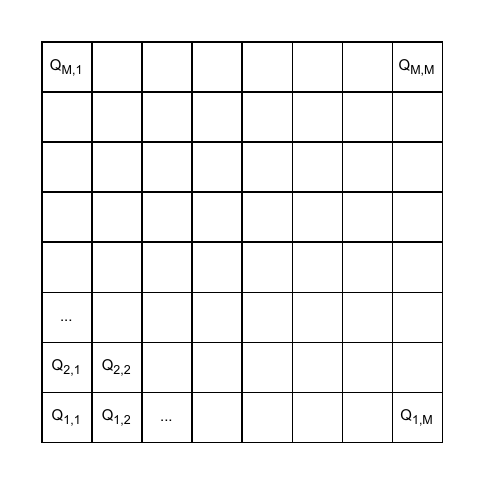}    
    \caption{Illustration of the tiling lemma.}
\end{figure}

\begin{proof}
Set $p=\max_{h,w}p_{h,w}$ and let $B \subset \mathbb{R}$ be any Borel set with the associated interval $I_B=[\inf B,\sup B]$. Let $y \in S^{N-1}$ be arbitrary. 
Then, consider the set 
\[
O:=\{x \in (0,1)^2 \mid y^Tg(x) \in I_B \}
\]
and we have 
\[
O=\bigcup_{h,w=1}^M (O \cap Q_{h,w})
\]
and since $Q_{h,w} \cap Q_{h',w'}$ is a null set for $(h,w)\neq (h',w')$, we have
\begin{align*}
\nu_y(B) &\leq \nu_y(I_B) = \vert O \vert  =\sum_{h,w=1}^{M} \vert O \cap Q_{h,w} \vert  \\
& \leq \sum_{h,w=1}^M \vert \{x \in Q_{h,w} \mid y^Tg(x) \in I_B \} \vert \\
&= \frac{1}{M^2}\sum_{h,w=1}^M \vert \{x \in (0,1)^2 \mid y^T(g\vert_{Q_{h,w}} \circ \tau_{h,w})(x) \in I_B \} \vert  \\
& \leq \frac{1}{M^2}\sum_{h,w=1}^M C_{h,w} \vert I_B \vert^{\frac{2}{p_{h,w}}}  \leq \bigg( 1+\sum_{h,w=1}^M C_{h,w} \bigg) \vert I_B \vert^{\frac{2}{p}},
\end{align*}
where for $\vert I_B \vert \leq 1$ we used $\vert I_B \vert^{\frac{2}{p_{h,w}}} \leq \vert I_B \vert^{\frac{2}{p}}$, while for $\vert I_B \vert > 1$ the bound $\vert O \vert \leq 1 \leq \vert I_B \vert^{\frac{2}{p}}$ is trivial.
Noting that $\vert I_B \vert^{\frac{2}{p}}=\mathrm{diam}(B)^{\frac{2}{p}}$ and using the same techniques as in the $(\impliedby)$ direction of the proof of Theorem \ref{Thm: Equivalence}, we see that $g$ satisfies the RDC with exponent
$p=\max_{h,w}p_{h,w}$ and constant $\tilde{C}:=\bigg( 1+\sum_{h,w=1}^M C_{h,w} \bigg)$.
\end{proof}

Our analysis, and in particular Theorem \ref{Thm: Equivalence},
reveals that the pushforward measure $\nu$ of $g$, equivalently all its marginals $\nu_y$, is the only invariant relevant for admitting a CFC. Furthermore, the equivalence
\[
\nu_y(B)=\int_{(0,1)^2} \mathbf{1}_B(y^Tg(x))\dx
\]
reveals that the marginals are occupation measures \cite{OccupationDensity}. However, in contrast to the theory of occupation measures, we do not require $\nu_y$ to come from a density $h_y$, i.e. to only have an absolutely continuous part with respect to $\mathcal{L}^1$ and no jump or Cantor part. For $N=1$ and $p = \frac{2\ln 3}{\ln 2} > 2$, taking $F$ to be the Cantor function yields an image whose pushforward measure $\nu$ consists exclusively of a Cantor part. The Fourier transform, denoted by the hat $\hat{\;}$, is a widely used technique -- not just in the theory of occupation measures -- to study the pushforward measure $\nu$ as a measure on $\mathbb{R}^N$.
To generate an abundance of images $g \in \mathcal{A}(p)$,  we can add suitable noise to any image and then use the Fourier transform of the pushforward measure of the noisy image to show that it must actually come from a density. Our noise model will be a fractional Gaussian field, which injects noise at every scale into our image $g$.

\begin{defn}[Fractional Gaussian field]\label{Def: FGF}
Let $0 < H < 1$. The \textbf{fractional Gaussian field} $\mathfrak{W} = (\mathfrak{W}_1, \dots, \mathfrak{W}_N)$ with \textbf{Hurst index} $H$ is the $\mathbb{R}^N$-valued centered Gaussian field on $\mathbb{R}^2$ whose components are i.i.d.\ with
\[
\mathfrak{W}_i(0) = 0, \qquad \mathbb{E}\big[\vert \mathfrak{W}_i(x) - \mathfrak{W}_i(x') \vert^2\big] = \vert x - x' \vert^{2H}
\]
for all $x,x' \in \mathbb{R}^2$.
\end{defn}
The theory surrounding these fields, including existence, can be found in \cite{fractionalgaussian}. In particular, we have $W:=\mathfrak{W}\vert_{(0,1)^2} \in L^{\infty}((0,1)^2,\mathbb{R}^N)$ almost surely. 

\begin{thm}[Noisy images are in $\mathcal{A}(2)$]\label{Thm: Noise makes images admissible}
Let $g \in L^{\infty}((0,1)^2,\mathbb{R}^N)$ and $W$ be the fractional Gaussian field (restricted to $(0,1)^2$) with Hurst index $H$ satisfying $NH<1$. Then, for every $\sigma>0$, we have almost surely
\[
g + \sigma W \in \mathcal{A}(2).
\]
\end{thm}

\begin{proof}
Let us set $\tilde{g}=g + \sigma W$ and we have $\tilde{g} \in L^{\infty}((0,1)^2,\mathbb{R}^N)$ almost surely. We define the Fourier transform of 
the pushforward measure of $\tilde{g}$ by
\[
\hat{\nu}(\xi)=\int_{(0,1)^2}e^{i\xi \cdot \tilde{g}(x)}\text{d}x.
\]

Recall that the characteristic function of a Gaussian random vector $Z \sim \mathcal{N}(\mu,\Sigma)$ satisfies
\[
\mathbb{E}[e^{i\xi \cdot Z}]=e^{i \xi \cdot \mu}e^{-\frac{1}{2} \xi^T \Sigma \xi},
\]
that the increments of $\tilde{g}$ satisfy
\[
\tilde{g}(x)-\tilde{g}(x') \sim \mathcal{N}(g(x)-g(x'), \sigma^2 \vert x-x' \vert^{2H}I),
\]
for $x,x' \in (0,1)^2$ and that 
\[
M_{s,N}:=\int_{\mathbb{R}^N}\vert u \vert^s e^{-\frac{1}{2} \vert u \vert^2} \text{d}u < \infty.
\]
Using these three facts and the change of variables $u=\sigma \vert x-x' \vert^{H} \xi$, we have 
\begin{align*}
 \mathbb{E} \bigg[\int_{\mathbb{R}^N}  \vert\xi\vert^s\,\big\vert\hat{\nu}(\xi)\big\vert^2\text{d}\xi \bigg]
&=\int_{\mathbb{R}^N}\vert\xi\vert^s\int_{(0,1)^2}\!\int_{(0,1)^2}\mathbb{E}\big[e^{i\xi \cdot (\tilde{g}(x)-\tilde{g}(x'))}\big]\text{d}x\text{d}x'\text{d}\xi\\
&=\int_{\mathbb{R}^N}\vert\xi\vert^s\int_{(0,1)^2}\!\int_{(0,1)^2}\underbrace{e^{i\xi \cdot (g(x)-g(x'))}}_{\vert \cdot \vert=1}\,e^{-\frac{1}{2}\sigma^2 \vert x-x'\vert^{2H}\vert\xi\vert^2}\text{d}x\text{d}x'\text{d}\xi\\
&\leq\int_{(0,1)^2}\int_{(0,1)^2}\bigg(\int_{\mathbb{R}^N}\vert\xi\vert^s\,e^{-\frac{1}{2}\sigma^2 \vert x-x'\vert^{2H}\vert\xi\vert^2}\text{d}\xi\bigg)\text{d}x\text{d}x'\\
&=\int_{(0,1)^2}\!\int_{(0,1)^2}\big(\sigma^2 \vert x-x'\vert^{2H}\big)^{-\frac{s+N}{2}}\bigg(\int_{\mathbb{R}^N}\vert u \vert^s\,e^{-\frac{1}{2}\vert u \vert^2}\text{d}u\bigg)\text{d}x\text{d}x'\\
&=\int_{(0,1)^2}\!\int_{(0,1)^2}M_{s,N}\,\big(\sigma^2 \vert x-x'\vert^{2H}\big)^{-\frac{s+N}{2}}\text{d}x\text{d}x'\\
&=\frac{M_{s,N}}{\sigma^{s+N}}\int_{(0,1)^2}\int_{(0,1)^2}\vert x-x'\vert^{-(s+N)H}\text{d}x\text{d}x'\\
&\leq C(s,\sigma,H)
\end{align*}
for some constant $C(s,\sigma,H)>0$. In the last step, we have used that the singular integral converges for $(s+N)H<2$, which we can achieve by setting 
\[
s \in (N,\frac{2}{H}-N),
\]
which is non-empty for $NH<1$. Since $\vert \hat{\nu}(\xi) \vert \leq 1$, we have that
\[
\int_{\vert \xi \vert \leq 1}(1+\vert \xi \vert^2)^{\frac{s}{2}}\vert \hat{\nu}(\xi) \vert^2\text{d}\xi
\]
is bounded by a constant depending only on $s$ and $N$, and $(1+\vert \xi \vert^2)^{\frac{s}{2}} \leq 2^{\frac{s}{2}}\vert \xi \vert^{s}$ for $\vert \xi \vert \geq 1$. We therefore have a bound on the full Sobolev $H^{\frac{s}{2}}$-norm, and Markov's inequality gives
\[
\mathbb{E} \big[ \Vert \nu \Vert_{H^{\frac{s}{2}}(\mathbb{R}^N)}^2 \big] < \infty
\implies
\Vert \nu \Vert_{H^{\frac{s}{2}}(\mathbb{R}^N)} < \infty
\]
almost surely. In particular, the tempered distribution $\nu$ is represented by a function $h \in H^{\frac{s}{2}}(\mathbb{R}^N)$, i.e.
\[
\langle \nu, \varphi \rangle = \int_{\mathbb{R}^N} h\varphi\,\text{d}x \qquad \text{for all } \varphi \in \mathcal{S}(\mathbb{R}^N).
\]
Equivalently, $\nu$ is absolutely continuous with respect to the Lebesgue measure with density $h$, for which we write $\nu = h\text{d}x$. 
Since $s>N$, i.e. $\frac{s}{2}>\frac{N}{2}$, the Sobolev embedding then yields for some constant $C>0$
\[
\Vert h \Vert_{L^{\infty}(\mathbb{R}^N)} \leq C \Vert h \Vert_{H^{\frac{s}{2}}(\mathbb{R}^N)} < \infty.
\]
Then, we have almost surely that 
\begin{itemize}
    \item $W$ is bounded and 
    \item the pushforward measure of $\tilde{g}$ admits a bounded density $h$.
\end{itemize} 
Consequently, the marginal densities $h_y$ of $\nu_y$ satisfy, for every $y \in S^{N-1}$ and almost every $t \in \mathbb{R}$,
\begin{align*}
\vert h_y(t) \vert &= \int_{ \{ y^Tx=t \}} h(x) \text{d}\mathcal{H}^{N-1}(x)
 \leq \Vert h \Vert_{L^{\infty}(\mathbb{R}^N)} \mathcal{H}^{N-1}\bigg( \{ y^Tx=t \} \cap B_{\tilde{R}}(0) \bigg) \\
 &\leq C \Vert h \Vert_{H^{\frac{s}{2}}(\mathbb{R}^N)} \mathcal{H}^{N-1}\bigg( \{x \in \mathbb{R}^{N-1} \mid \vert x \vert \leq 1 \}\bigg) \tilde{R}^{N-1}  < \infty,
\end{align*}
where $B_{\tilde{R}}(0):=\{x \in \mathbb{R}^N \mid \vert x \vert \leq \tilde{R} \}$ is chosen such that $\tilde{g}((0,1)^2) \subset B_{\tilde{R}}(0) $.
Since the bound is independent of $y \in S^{N-1}$, we can take the supremum and applying the inequality chain \eqref{Estimate: Density gives RDC} gives the RDC with exponent $p=2$.
\end{proof}
The only properties of the Gaussian field $W$ that the proof uses are boundedness and the increments $\mathbb{E}\big[|W_i(x) - W_i(x')|^2 \big] = \eta(|x-x'|)$, which enters as $e^{-\frac{1}{2}\sigma^2 \eta(|x-x'|)|\xi|^2}$. After the change of variables $u = \sigma \eta(|x-x'|)^{\frac{1}{2}}\xi$, the argument goes through verbatim for any positive $\eta$ (of course provided that the associated Gaussian field exists) such that 
\[\int_{(0,1)^2}\int_{(0,1)^2} \eta(|x-x'|)^{-\frac{s+N}{2}}\,\mathrm{d}x\,\mathrm{d}x' < \infty
\] 
for some $s > N$. For $\eta(\alpha) = \alpha^{2H}$ this is precisely $NH < 1$, so the restriction on the Hurst index is a consequence of our specific noise model and avoids more complicated integral evaluations, such as for the Matérn covariance.

As a corollary of the previous theorem, we get that images satisfying the CFC are dense in $L^{\infty}((0,1)^2,\mathbb{R}^N)$. In particular, the set is 
non-empty.

\begin{cor}[Density of non-degenerate images]\label{Thm: Admissible maps are dense - noise version}
We have
\[
\overline{\mathcal{A}(2)}^{L^{\infty}} = L^{\infty}((0,1)^2,\mathbb{R}^N),
\]
where $\overline{\mathcal{A}(2)}^{L^{\infty}}$ denotes the closure of $\mathcal{A}(2)$ in $L^{\infty}((0,1)^2,\mathbb{R}^N)$.
\end{cor}

\begin{proof}
Take any fractional Gaussian field $W$ with Hurst index $H$ satisfying $NH<1$. By the previous Theorem \ref{Thm: Noise makes images admissible}, we have
\[
\mathbb{P}(g+ \tfrac{1}{n} W \in \mathcal{A}(2))=1
\]
for every $n \in \mathbb{N}$. Hence
\[
\mathbb{P}\bigg( \bigcap_{n \in \mathbb{N}} \{ g+ \tfrac{1}{n} W \in \mathcal{A}(2) \} \bigg)=1.
\]
Fix any such $\omega \in \Omega$ in the underlying probability space $\Omega$. Hence,
\[
\Vert g- (g + \tfrac{1}{n}W(\omega)) \Vert_{L^{\infty}((0,1)^2,\mathbb{R}^N)} \leq \frac{1}{n}\Vert W(\omega) \Vert_{L^{\infty}((0,1)^2,\mathbb{R}^N)} \to 0
\]
as $n \to \infty $. This gives us $\overline{\mathcal{A}(2)}^{L^{\infty}} \supset L^{\infty}((0,1)^2,\mathbb{R}^N)$, while the other inclusion is covered by the definition of $\mathcal{A}(2)$.
\end{proof}

The corollary above also shows that $\mathcal{A}(p)$ is not closed in the $L^{\infty}$ topology for any $p \geq 2$, as we can approximate any constant function arbitrarily well, but it does not contain any open set in the $L^{\infty}$ topology.

We would like to close this section by noting that if we take $X \sim U((0,1)^2)$ to be a uniformly distributed random variable on $(0,1)^2$, we can adopt a probabilistic viewpoint of the CFC/RDC by studying the random variable $Z=g(X)$ via a quantitative growth bound on its Lévy concentration function \cite{ConcentrationFunctions} $Q(y^TZ,r):= \sup_{t \in \mathbb{R}} \mathbb{P}(y^TZ \in [t,t+r])$ for $r>0$, i.e.
\[
\sup_{y \in S^{N-1}}Q(y^TZ,r) \leq Cr^{\frac{2}{p}}
\]
for some constant $C>0$. As in  \cite{Mourtada_2022}, we use this as a condition on $g$, but we actually show that any image $g$ with suitable noise has a quantitative (depending on the scale) covariance floor and hence satisfies this condition.
However, images are not random variables in spectral space and the underlying spatial-spectral structure of $g \in L^{\infty}((0,1)^2,\mathbb{R}^N)$ matters when we use the Gaussian Mumford--Shah model to find a segmentation of the image. While the data term is invariant under spatial re-arrangements, the perimeter term is not. None of the three similar concepts -- Frostman's lemma, occupation measures and Lévy concentration --
requires uniformity over all directions $y \in S^{N-1}$ for $N>1$.

\section{Lower bounds for the Gaussian Mumford--Shah model}
For this section, we consider the set
\[
X:=P_k \times (\mathbb{R}^N)^k \times \SPD(N)^k.
\]
We equip $P_k$ with the $L^1$-topology, i.e. for each $1 \leq l \leq k$, we identify each $O_l \subset (0,1)^2$ (up to null sets) with its characteristic function $\mathbf{1}_{O_l}$. We endow $P_k$ with the metric 
\[
d_{P_k}(\mathcal{O},\mathcal{O}'):=\sum_{l=1}^k \Vert \mathbf{1}_{O_l}- \mathbf{1}_{O_l'} \Vert_{L^1((0,1)^2)} \quad \text{for } \mathcal{O},\mathcal{O}' \in P_k.
\]
We equip $\SPD(N)$ with the metric coming from the Frobenius norm $\vert \cdot \vert_F$ and $\mathbb{R}^N$ with the standard Euclidean metric.
Furthermore, we note that if $\vert O_l \vert=0$, we can set $(\mu_l,\Sigma_l)$ to any value without changing the energy of the Gaussian Mumford--Shah model $E$ \eqref{Def: GMS}. We choose $m_{O_l}=0$ and $C_{O_l}=I$ in this case.

In order to find lower bounds for the GMS, we would like to use the fact that the indicator functions $f(g(x) \mid \mu_l,\Sigma_l)$ ($1 \leq l \leq k$) of the Gaussian Mumford--Shah model are negative log-likelihoods of multivariate Gaussians $\mathcal{N}(\mu_l,\Sigma_l)$ instead of arguing from a purely optimization/multivariable calculus perspective. As a consequence, we need the following subset of the space $\mathcal{P}(\mathbb{R}^N)$ of probability distributions on $\mathbb{R}^N$ coming from a positive density
\[
\mathcal{P}_{++}(\mathbb{R}^N):=\bigg\{Q \in \mathcal{P}(\mathbb{R}^N) \bigg\vert \frac{\text{d}Q}{\text{d}\mathcal{L}^N} \in C^0(\mathbb{R}^N,(0,\infty)) \bigg\},
\]
i.e. having a strictly positive Radon--Nikodym derivative.
\begin{defn}[Information theoretic quantities]
For $Q \in \mathcal{P}_{++}(\mathbb{R}^N)$ with density $q=\frac{\text{d}Q}{\text{d}\mathcal{L}^N}$ and $R \in \mathcal{P}(\mathbb{R}^N)$, we define the \textbf{entropy} by 
\[
H(Q):=-\int_{\mathbb{R}^N} q(z) \ln q(z)\text{d}z,
\]
the \textbf{cross entropy} by 
\[
H(R,Q):=- \int_{\mathbb{R}^N} \ln q(z) \text{d}R(z),
\]
and if in addition we have $R \in \mathcal{P}_{++}(\mathbb{R}^N)$ with density $r=\frac{\text{d}R}{\text{d}\mathcal{L}^N}$, the \textbf{Kullback--Leibler divergence} by
\[
\mathrm{KL}(R \mid Q):= H(R,Q)-H(R)=\int_{\mathbb{R}^N}r(z) \ln \frac{r(z)}{q(z)}\text{d}z.
\]
\end{defn}
The following proposition relates lower bounds of the Gaussian Mumford--Shah model cleanly to information theoretic quantities:  

\begin{prop}[Information theoretic lower bound]\label{Prop: Information theoretic lower bound}
Let $g \in \mathcal{A}(p)$ for $p \geq 2$. Then, for the Gaussian Mumford--Shah model \eqref{Def: GMS}, we have
\begin{equation}\label{Inequality: Information theoretic lower bound}
\begin{split}
E[\mathcal{O},\mu,\Sigma] &\geq \sum_{l=1}^k \bigg( \vert O_l \vert H(\mathcal{N}(m_{O_l},C_{O_l}))+ \lambda \mathrm{Per}(O_l,(0,1)^2) \bigg) \\
& = \sum_{l=1}^k \bigg( \frac{1}{2} \vert O_l \vert \ln \big( (2 \pi e)^N \det C_{O_l} \big) + \lambda \mathrm{Per}(O_l,(0,1)^2) \bigg)
\end{split}
\end{equation}
for all $(\mathcal{O},\mu,\Sigma) \in X$ with equality if $\mu=(m_{O_1}, \dots,m_{O_k})$
 and $\Sigma=(C_{O_1}, \dots,C_{O_k})$.
\end{prop}
\begin{proof}
Let $(\mathcal{O},\mu,\Sigma) \in X$ be arbitrary and consider all $O_l \subset (0,1)^2$ with $\vert O_l \vert >0$. We note that the normalized pushforward measures 
$P_l=\frac{1}{\vert O_l \vert}g_{\sharp}\big(\mathcal{L}^2\big\vert_{O_l}\big)$ are probability measures with finite mean $m_{O_l}$ 
and finite positive definite covariance $C_{O_l} \succ 0$, where positive definiteness is guaranteed by $g \in \mathcal{A}(p)$. Let $q(z \mid \mu_l, \Sigma_l)=\frac{1}{(2\pi)^{\frac{N}{2}}}\frac{1}{(\det \Sigma_l)^{\frac{1}{2}}} e^{-\frac{1}{2}\vert z- \mu_l \vert^2_{\Sigma_l}}$ denote the density of a multivariate Gaussian. \linebreak Since $-\ln q(z \mid \mu_l, \Sigma_l)$ is a quadratic polynomial, the cross entropy $H(P_l,\mathcal{N}(\mu_l,\Sigma_l))=-\int_{\mathbb{R}^N} \ln q(z \mid \mu_l, \Sigma_l) \text{d}P_l(z)$ depends on $P_l$ only through its mean and covariance. Therefore, we have 
\begin{equation*}
\begin{split}
H(P_l,\mathcal{N}(\mu_l,\Sigma_l))&=H(\mathcal{N}(m_{O_l},C_{O_l}),\mathcal{N}(\mu_l,\Sigma_l)) \\
&=H(\mathcal{N}(m_{O_l},C_{O_l}))+\underbrace{\mathrm{KL}(\mathcal{N}(m_{O_l},C_{O_l}) \mid \mathcal{N}(\mu_l,\Sigma_l))}_{\geq 0} \\
&\geq H(\mathcal{N}(m_{O_l},C_{O_l}))=\frac{1}{2}\ln \big( (2 \pi e)^N \det C_{O_l} \big).
\end{split}
\end{equation*}
We note that $\mathrm{KL}(\mathcal{N}(m_{O_l},C_{O_l}) \mid \mathcal{N}(\mu_l,\Sigma_l))$ vanishes exactly if $\mu_l=m_{O_l}$ and $\Sigma_l=C_{O_l}$. Recalling that $f(z \mid \mu_l, \Sigma_l)=-\ln q(z \mid \mu_l, \Sigma_l)$, we can rewrite the energy $E$ \eqref{Def: GMS} using the weighted pushforward measure $P_l$
\begin{align*}
E[\mathcal{O},\mu,\Sigma]&=\sum_{l=1}^k \bigg( -\int_{O_l} \ln q(g(x) \mid \mu_l, \Sigma_l)\dx + \lambda \mathrm{Per}(O_l,(0,1)^2) \bigg) \\
&=\sum_{l=1}^k \bigg( - \vert O_l \vert \int_{\mathbb{R}^N} \ln q(z \mid \mu_l, \Sigma_l)\text{d}P_l(z) + \lambda \mathrm{Per}(O_l,(0,1)^2) \bigg) \\
&=\sum_{l=1}^k \bigg( \vert O_l \vert H(P_l,\mathcal{N}(\mu_l,\Sigma_l)) + \lambda \mathrm{Per}(O_l,(0,1)^2) \bigg) \\
&\geq \sum_{l=1}^k \bigg( \vert O_l \vert H(\mathcal{N}(m_{O_l},C_{O_l})) + \lambda \mathrm{Per}(O_l,(0,1)^2) \bigg) \\
&=\sum_{l=1}^k \bigg( \vert O_l \vert  \int_{\mathbb{R}^N}f(z \mid m_{O_l},C_{O_l}) q(z \mid m_{O_l}, C_{O_l}) \text{d}z + \lambda \mathrm{Per}(O_l,(0,1)^2) \bigg) \\
&= \sum_{l=1}^k \bigg( \frac{1}{2} \vert O_l \vert \ln \big( (2 \pi e)^N \det C_{O_l} \big)  + \lambda \mathrm{Per}(O_l,(0,1)^2) \bigg).
\end{align*}
\end{proof}
Note that in the lower bound \eqref{Inequality: Information theoretic lower bound}, we have profiled out the statistical variables $(\mu,\Sigma)$, and we are therefore left with a lower bound only depending on the partition. The inequality in \eqref{Inequality: Information theoretic lower bound} shows that the energy of the Gaussian Mumford--Shah model is always bounded from below by a sum of entropies times segment size. In fact, if we choose $\mu_l=m_{O_l}$ and $\Sigma_l=C_{O_l}$ for $1 \leq l \leq k$ as a default, the inequality in \eqref{Inequality: Information theoretic lower bound} becomes sharp as $\mathrm{KL}(\mathcal{N}(m_{O_l},C_{O_l}) \mid \mathcal{N}(m_{O_l},C_{O_l}))=0$ and minimizing the Gaussian Mumford--Shah model is equivalent to minimizing entropy times segment size per segment plus a regularizer. However, even though these two forms have the same minimizing partition $\mathcal{O}$, computing this minimizing partition with $(\mu_l,\Sigma_l)=(m_{O_l},C_{O_l})$ for $1 \leq l \leq k$ tied to the partition blocks us from using convex relaxations \cite{ConvexificationChoices} to approximate minimizing partitions numerically.

\begin{thm}[Lower bounds]\label{Thm: Lower bounds}
Let $g \in \mathcal{A}(p)$ for $p \geq 2$. Then, there exists a constant $K_0>0$ such that the energy $E$ \eqref{Def: GMS} of the Gaussian Mumford--Shah model is bounded below, i.e.
\begin{equation}\label{Lower bound}
E[\mathcal{O},\mu,\Sigma] \geq -K_0 + \lambda \sum_{l=1}^k \mathrm{Per}(O_l,(0,1)^2) \geq -K_0  
\end{equation}  
for all $(\mathcal{O},\mu,\Sigma) \in X$. 
\end{thm}
\begin{proof}
Let $g \in \mathcal{A}(p)$ for $p \geq 2$. Let us introduce the function 
\begin{equation}\label{Def: psi-function t*ln(t)}
\psi:[0,1] \to \mathbb{R}, \quad t \mapsto \begin{cases}
t \ln t &\text{ for } 0 < t \leq 1 \\
0 & \text{ for } t=0 \end{cases},
\end{equation}
which is continuous on $[0,1]$ and bounded below by $-\frac{1}{e}$.
We then use the information theoretic lower bound from the previous Proposition \ref{Prop: Information theoretic lower bound} and, since $g \in \mathcal{A}(p)$, there exists $C>0$ such that $C_{O_l} \succeq C \vert O_l \vert^p I$ for all $1 \leq l \leq k$, and we estimate
\begin{align*}
E[\mathcal{O},\mu,\Sigma] &\geq \sum_{l=1}^k \bigg( \frac{1}{2} \vert O_l \vert \ln \big( (2 \pi e)^N \det C_{O_l} \big)  + \lambda \mathrm{Per}(O_l,(0,1)^2) \bigg) \\
&\geq \sum_{l=1}^k \bigg( \frac{1}{2} \vert O_l \vert \ln \big( (2 \pi e)^N \det (C \vert O_l \vert^p I) \big)  + \lambda \mathrm{Per}(O_l,(0,1)^2) \bigg) \\
&=\sum_{l=1}^k \bigg( \frac{N\ln(2\pi e C)}{2} \vert O_l \vert + \frac{pN}{2} \vert O_l \vert  \ln \vert O_l \vert  + \lambda \mathrm{Per}(O_l,(0,1)^2) \bigg) \\
&\geq - \frac{N}{2} \vert \ln(2\pi e C) \vert -k\frac{pN}{2e}+ \lambda \sum_{l=1}^k \mathrm{Per}(O_l,(0,1)^2) \\
&\geq -K_0+\lambda \sum_{l=1}^k \mathrm{Per}(O_l,(0,1)^2) \geq -K_0
\end{align*}
for $K_0=\frac{N}{2} \vert \ln(2\pi e C) \vert +k\frac{pN}{2e}>0$, since $\psi(\vert O_l \vert)=\vert O_l \vert \ln \vert O_l \vert \geq -\frac{1}{e}$ and $\sum_{l=1}^k \vert O_l \vert=1$ by virtue of being a partition. 
\end{proof}
The proof consumes the covariance floor $C\vert O\vert^p I$ only through the term $\vert O_l\vert \ln\vert O_l\vert$. Any non-decreasing $\phi:(0,1]\to(0,\infty)$ with $C_O \succeq \phi(\vert O\vert) I$ for all measurable $O\subset(0,1)^2$ with $\vert O\vert>0$ and
\[
\inf_{0<s\le 1} s\ln\phi(s) > -\infty
\]
ensures the existence of a lower bound. In Theorem \ref{Thm: Equivalence}, this would translate to a modulus of continuity condition on the cumulative distributions.

We will postpone the existence proof to the next section, namely, Theorem \ref{Thm: Convergence of minima}, to use \enquote{unified} technical lemmas and propositions for both $E^{(\epsilon)}$ and $E$, even though we could now apply the direct method of the calculus of variations \cite{CoVRindler}.

\section{Applications}
We can now use non-trivial but standard analysis to prove powerful theorems, where the biggest missing ingredient was the lower bound given by Theorem \ref{Thm: Lower bounds}. 

\subsection{Existence theory and convergence rates}
For any $\mathcal{O} \in P_k$, let us consider the set of active labels
\[
A(\mathcal{O}):=\{l \in \{1,\dots,k \} \mid \vert O_l \vert >0 \}.
\]
We also use the shorthand notation $S:=(\mathbb{R}^N)^k \times \SPD(N)^k$ for the domain of the statistical variables. It is now convenient to profile out the statistical variables, as we are mainly interested in the minimizing partitions of the GMS and pGMS. 

\begin{lem}[Closed forms for partition dependent energy]\label{Lem: Partition dependant energies}
Let $g \in \mathcal{A}(p)$ for $p \geq 2$ and $\epsilon \in [0,1]$. Then, the \textbf{reduced energies} $\tilde{E}^{(\epsilon)}$
satisfy 
\begin{align*}
\tilde{E}^{(\epsilon)}[\mathcal{O}]:&=\inf_{(\mu,\Sigma) \in S}E^{(\epsilon)}[(\mathcal{O},\mu,\Sigma)] \\
&=\sum_{l \in A(\mathcal{O})} \bigg( \frac{\vert O_l \vert}{2}\ln \big((2\pi e)^N \det(C_{O_l}+\epsilon I) \big) + \lambda \mathrm{Per}(O_l,(0,1)^2)  \bigg),
\end{align*}
where $\tilde{E}=\tilde{E}^{(0)}$, with the infimum being uniquely attained at $(\hat{\mu}^{(\epsilon)},\hat{\Sigma}^{(\epsilon)}) \in S$ given by 
\[
\hat{\mu}_l^{(\epsilon)}=m_{O_l} \quad \text{and} \quad \hat{\Sigma}_l^{(\epsilon)}=C_{O_l}+\epsilon I
\]
for $l \in A(\mathcal{O})$ and by any $(\hat{\mu}_l^{(\epsilon)},\hat{\Sigma}_l^{(\epsilon)}) \in \mathbb{R}^N \times \SPD(N)$ for $l \notin A(\mathcal{O})$. Furthermore,
\[
\inf_{(\mathcal{O},\mu,\Sigma) \in X } E^{(\epsilon)}[(\mathcal{O},\mu,\Sigma)]=\inf_{\mathcal{O} \in P_k} \tilde{E}^{(\epsilon)}[\mathcal{O}].
\]
\end{lem}
\begin{proof}
For $l \in A(\mathcal{O})$, completing the square gives 
\begin{align*}
&\int_{O_l}f(g(x) \mid \mu_l,\Sigma_l)\dx+\frac{\epsilon}{2}\vert O_l \vert \mathrm{tr}  \Sigma_l^{-1}  \\
=&\;\frac{1}{2}\vert O_l \vert \bigg( \mathrm{tr}(\Sigma_l^{-1}(C_{O_l}+\epsilon I))+ \vert m_{O_l}-\mu_l \vert_{\Sigma_l}^2+ \ln \det \Sigma_l + N \ln (2\pi)\bigg),
\end{align*}
which is minimized precisely at $\mu_l=m_{O_l}$ and $\Sigma_l=C_{O_l}+ \epsilon I \succ 0$ using the same arguments as for the log-likelihood of a multivariate Gaussian \cite{AndersenMultivariate}. This holds trivially for $\epsilon>0$, and for $\epsilon=0$ we use $g \in \mathcal{A}(p)$. For $l \notin A(\mathcal{O})$, the $l$-th summand of $E^{(\epsilon)}$ vanishes for every $(\mu_l,\Sigma_l)$. We also have $\vert O_l \vert=0 \implies \mathrm{Per}(O_l,(0,1)^2)=0$. Substituting this into $E^{(\epsilon)}$ yields the desired identity and the equality of the infima follows, since the infimum over the statistical variables $(\mu,\Sigma) \in S$ is actually attained, uniquely so in the variables of the active labels.
\end{proof}

We state the main theorem of this section, with $E^{(0)}:=E$, which is of course closely related to $\Gamma$-convergence, denoted by $\GC$. We refer to \cite{Braides} for all the necessary definitions, lemmas and background.

\begin{thm}[Convergence of minima]\label{Thm: Convergence of minima}
Let $g \in \mathcal{A}(p)$ for $p \geq 2$. For the energies \eqref{Def: pGMS} with $E^{(0)}:=E$ the following holds: 
\begin{itemize}
    \item[(i)](Existence of minimizers) For every $\epsilon \in [0,1]$, there exists a minimizer $\hat{x}^{(\epsilon)}=(\hat{\mathcal{O}}^{(\epsilon)},\hat{\mu}^{(\epsilon)},\hat{\Sigma}^{(\epsilon)}) \in X$, i.e.
            \begin{equation}\label{Eq: Existence of minimizers for all epsilon}
            \inf_{x \in X} E^{(\epsilon)}[x]=E^{(\epsilon)}[\hat{x}^{(\epsilon)}].
            \end{equation}
    \item[(ii)](Convergence rate) For all $\epsilon \in (0,1]$ and all $\mathcal{O} \in P_k$, we have
              \begin{equation}\begin{split}\label{Eq: Convergence of energy}
            \bigg\vert \tilde{E}^{(\epsilon)}[\mathcal{O}]- \tilde{E}[\mathcal{O}] \bigg\vert \leq \vert A(\mathcal{O}) \vert L \epsilon^{\frac{1}{p}} &\to 0 \\
             \text{and} \quad \bigg\vert \min_{\mathcal{O} \in P_k}\tilde{E}^{(\epsilon)}[\mathcal{O}]- \min_{\mathcal{O} \in P_k} \tilde{E}[\mathcal{O}] \bigg\vert \leq k L \epsilon^{\frac{1}{p}} &\to 0
            \end{split}\end{equation}
            as $\epsilon \to 0$ with constant $L>0$ independent of $\epsilon$ and $\mathcal{O}$.
    \item[(iii)](Convergence of minimizers)  Every family of minimizers $(\hat{x}^{(\epsilon)})_{\epsilon \in (0,1]}$ has a subsequence $\epsilon_j \to 0$ such that for a minimizer $\hat{x}=(\hat{\mathcal{O}},\hat{\mu},\hat{\Sigma}) \in X$ of $E$ we have
\begin{equation}\begin{split}\label{Eq: Convergence of the partitions and mean/cov for active labels}
   &d_{P_k}(\hat{\mathcal{O}}^{(\epsilon_j)},\hat{\mathcal{O}}) \to 0, \text{ and } \\
   &\vert \hat{\mu}_l^{(\epsilon_j)}-\hat{\mu}_l \vert + \vert \hat{\Sigma}_l^{(\epsilon_j)}-\hat{\Sigma}_l \vert_F \to 0 \text{ for all } l \in A(\hat{\mathcal{O}}),
\end{split}\end{equation}
as $j \to \infty$.
\end{itemize}
\end{thm}
One may alternatively extend $E$ to positive semi-definite $\Sigma_l$ on empty segments (i.e. $\vert O_l \vert = 0$). After a further extraction of a subsequence, the minimizers then converge in all components of the extended space, with respect to the same topology. Since the limiting statistical variables of vanishing segments are arbitrary and energetically irrelevant, we prefer the formulation above.

Lower semi-continuity of the energies is also a property tied to $g$ having a covariance floor:
\begin{lem}[Lower semi-continuity]\label{Lem: LSC of the energies}
Let $g \in \mathcal{A}(p)$. Then, for $\epsilon \in [0,1]$, the energies $\tilde{E}^{(\epsilon)}$ are lower semi-continuous 
on $(P_k,d_{P_k})$.
\end{lem}
\begin{proof}
Let $d_{P_k}(\mathcal{O}^{(j)},\mathcal{O}) \to 0$ and write, for a single label,
\[
D_l^{(\epsilon)}[\mathcal{O}^{(j)}]:=\frac{\vert O_l^{(j)} \vert}{2}\ln \big((2\pi e)^N \det(C_{O_l^{(j)}}+\epsilon I) \big).
\]

\textbf{Case 1 ($l \in A(\mathcal{O})$):} In this case, we have $\vert O_l^{(j)} \vert \geq \frac{\vert O_l \vert}{2} >0$ for $j$ large. Since $g$ and $g \otimes g$ are bounded,
\[
\bigg\vert \int_{O_l^{(j)}} g(x) \dx - \int_{O_l} g(x) \dx \bigg\vert \leq \Vert g \Vert_{L^{\infty}} \Vert \mathbf{1}_{O_l^{(j)}}-\mathbf{1}_{O_l} \Vert_{L^1((0,1)^2)},
\]
and analogously for $g\otimes g$ and the measures, which gives $m_{O_l^{(j)}}\to m_{O_l}$ and $C_{O_l^{(j)}} \to C_{O_l}\succeq C \vert O_l \vert^p I \succ 0$. Since $\det$ is continuous and the limit is non-singular, we get $D_l^{(\epsilon)}[\mathcal{O}^{(j)}] \xrightarrow{j \to \infty} D_l^{(\epsilon)}[\mathcal{O}]$. \\

\textbf{Case 2 ($l \notin A(\mathcal{O})$):}
We can use the inequality 
\[
C \vert O_l^{(j)} \vert^p I \preceq C_{O_l^{(j)}}+ \epsilon I \preceq(4 \Vert g \Vert_{L^{\infty}}^2+1)I,
\]
the function $\psi$ \eqref{Def: psi-function t*ln(t)} and the monotonicity of $\ln \det$ with respect to the Loewner order to deduce that 
\begin{align*}
\frac{N \ln (2\pi e C)}{2} \vert O_l^{(j)} \vert+  \frac{Np}{2}\psi(\vert O_l^{(j)} \vert)
\leq 
D_l^{(\epsilon)}[\mathcal{O}^{(j)}]
\leq \frac{N \ln \big( 2 \pi e (4 \Vert g \Vert_{L^{\infty}}^2+1) \big)}{2}  \vert O_l^{(j)} \vert.
\end{align*}
We note that both sides tend to $0$ as $\vert O_l^{(j)} \vert \to \vert O_l \vert=0$ and neither depends on $\epsilon$. Hence $D_l^{(\epsilon)}[\mathcal{O}^{(j)}] \to 0$, which is the contribution of the label $l$ in the limit. \\

The two previous cases combined give that the data term $\sum_{l} D_l^{(\epsilon)}[\cdot]$ is continuous for all $\epsilon \in [0,1]$ with respect to $(P_k,d_{P_k})$. The perimeter term is lower semi-continuous with respect to the $L^1$-topology \cite{ambrosio2000functions} and hence it is also lower semi-continuous on $(P_k,d_{P_k})$. As a result, $\tilde{E}^{(\epsilon)}$ is lower semi-continuous.
\end{proof}

We immediately get 
\begin{prop}[$\Gamma$-convergence of the pGMS]\label{Prop: GC of pGMS}
Let $g \in \mathcal{A}(p)$. We have 
\[
\GC_{\epsilon \to 0} \tilde{E}^{(\epsilon)}=\tilde{E}
\]
on $(P_k,d_{P_k})$.
\end{prop}
\begin{proof}
Let $\epsilon_j \to 0$ be any sequence. \\
\textbf{Step 1 ($\liminf$-inequality):} Take any sequence $\{ \mathcal{O}^{(j)} \} \subset P_k$ with $d_{P_k}(\mathcal{O}^{(j)},\mathcal{O}) \to 0$. By monotonicity of $\ln \det$ with respect to the Loewner order, we have $\tilde{E}^{(\epsilon_j)} \geq \tilde{E}$ pointwise, and hence
\[
\liminf_{j \to \infty } \tilde{E}^{(\epsilon_j)}[\mathcal{O}^{(j)}] \geq \liminf_{j \to \infty } \tilde{E}[\mathcal{O}^{(j)}] 
\geq \tilde{E}[\mathcal{O}],
\]
due to the previous Lemma \ref{Lem: LSC of the energies}, which proves the $\liminf$-inequality. \\
\textbf{Step 2 (recovery sequence):} For any $\mathcal{O} \in P_k$, the constant sequence $\mathcal{O}^{(j)}:=\mathcal{O}$ serves as a recovery sequence: labels $l \notin A(\mathcal{O})$ contribute zero to both $\tilde{E}^{(\epsilon_j)}[\mathcal{O}]$ and $\tilde{E}[\mathcal{O}]$, while for $l \in A(\mathcal{O})$ we have $C_{O_l}\succ 0$ and hence the (pointwise) convergence of 
\[
\lim_{\epsilon \to 0} \ln \det (C_{O_l}+\epsilon I)=\ln \det C_{O_l}.
\]
Summing over $l=1,\dots,k$ gives $\lim_{j \to \infty} \tilde{E}^{(\epsilon_j)}[\mathcal{O}]=\tilde{E}[\mathcal{O}]$. 

These two conditions imply $\Gamma$-convergence \cite{Braides}, i.e. $\GC_{\epsilon \to 0}\tilde{E}^{(\epsilon)}=\tilde{E}$.
\end{proof}

We are now ready to prove the main Theorem \ref{Thm: Convergence of minima}:
\begin{proof}
Let $\mathcal{O}^{\flat}=((0,1)^2,\emptyset, \dots,\emptyset)$ and $\tilde{E}^{\flat}:= \vert \tilde{E}^{(1)}[\mathcal{O}^{\flat}] \vert $ and $K_0>0$ the constant from Theorem \ref{Thm: Lower bounds}. Then, we set
\begin{equation}\label{Def: Set K (equi)}
K:= \bigg\{ \mathcal{O} \in P_k \bigg\vert \sum_{l=1}^k \mathrm{Per}(O_l,(0,1)^2)\leq \frac{\tilde{E}^{\flat}+K_0+1}{\lambda}  \bigg\},    
\end{equation}
which is compact in $(P_k,d_{P_k})$ because of the compact embedding $BV((0,1)^2) \subset \subset L^1((0,1)^2)$ \cite{ambrosio2000functions}. The partition constraints are also preserved by $L^1$-convergence and the perimeter bound is preserved by the lower semi-continuity of the perimeter under $L^1$-convergence. Using $\tilde{E}^{(\epsilon)} \geq \tilde{E}$ for $\epsilon>0$ together with Theorem \ref{Thm: Lower bounds}, i.e. $\tilde{E}^{(\epsilon)}[\mathcal{O}] \geq -K_0 + \lambda \sum_{l=1}^k \mathrm{Per}(O_l,(0,1)^2)$, we obtain
\[
\bigg\{ \mathcal{O} \in P_k \bigg\vert \tilde{E}^{(\epsilon)}[\mathcal{O}] \leq  \tilde{E}^{(\epsilon)}[\mathcal{O}^{\flat}] +1 \bigg\} \subset K
\]
and $K$ does not depend on $\epsilon$. For every $\epsilon \in [0,1]$, we can therefore choose a minimizing sequence $(\mathcal{O}^{(\epsilon,n)})_{n \in \mathbb{N}} \subset K$ of $\tilde{E}^{(\epsilon)}$,
which, by compactness, has a subsequence with limit $\hat{\mathcal{O}}^{(\epsilon)} \in K$. Using the lower semi-continuity of $\tilde{E}^{(\epsilon)}$, i.e. $\tilde{E}^{(\epsilon)}[\hat{\mathcal{O}}^{(\epsilon)}] \leq \liminf_{n \to \infty} \tilde{E}^{(\epsilon)}[\mathcal{O}^{(\epsilon,n)}]=\inf_{\mathcal{O} \in P_k}\tilde{E}^{(\epsilon)}[\mathcal{O}]$, and the statistical variables $(\hat{\mu}^{(\epsilon)},\hat{\Sigma}^{(\epsilon)})$ as in Lemma \ref{Lem: Partition dependant energies}, we infer that
\[
\inf_{x \in X} E^{(\epsilon)}[x]= \inf_{\mathcal{O} \in P_k} \tilde{E}^{(\epsilon)}[\mathcal{O}]=\tilde{E}^{(\epsilon)}[\hat{\mathcal{O}}^{(\epsilon)}]=E^{(\epsilon)}[\hat{x}^{(\epsilon)}]
\]
for $\hat{x}^{(\epsilon)}=(\hat{\mathcal{O}}^{(\epsilon)}, \hat{\mu}^{(\epsilon)}, \hat{\Sigma}^{(\epsilon)})$, which is the first statement \eqref{Eq: Existence of minimizers for all epsilon}. \\
To prove the convergence rate for any $\mathcal{O} \in P_k$, which in particular holds for $\hat{\mathcal{O}}^{(\epsilon)}$, we note by definition and Lemma \ref{Lem: Partition dependant energies},
\[
\tilde{E}^{(\epsilon)}[\mathcal{O}]-\tilde{E}[\mathcal{O}]=\sum_{l \in A(\mathcal{O})} \frac{\vert O_l \vert}{2} \ln \frac{\det (C_{O_l}+ \epsilon I)}{\det C_{O_l}}.
\]
Using the assumption $g \in \mathcal{A}(p)$, we note that $C_{O_l}$ is invertible and all eigenvalues $\kappa_{l,1}, \dots,\kappa_{l,N}$ are bounded below by $C \vert O_l \vert^p$. Hence
\begin{equation*}
\begin{split}
\tilde{E}^{(\epsilon)}[\mathcal{O}]-\tilde{E}[\mathcal{O}]&= \sum_{l \in A(\mathcal{O})} \frac{\vert O_l \vert}{2} \sum_{i=1}^N \ln \bigg(1+\frac{\epsilon}{\kappa_{l,i}} \bigg) \\
& \leq  \sum_{l \in A(\mathcal{O})} \underbrace{\frac{N \vert O_l \vert}{2} \ln\bigg(1+\frac{\epsilon}{C \vert O_l \vert^p}\bigg)}_{:=\varphi_{\epsilon}(\vert O_l \vert)}.
\end{split}
\end{equation*}
Substituting $u_l=\frac{\epsilon}{C \vert O_l \vert^p}$, we have $\varphi_{\epsilon}(\vert O_l \vert)= \frac{N}{2C^{\frac{1}{p}}} \epsilon^{\frac{1}{p}} \frac{\ln(1+u_l)}{u_l^{\frac{1}{p}}}$. Since $p>1$, we have $0<M_p=\sup_{u>0} \frac{\ln(1+u)}{u^{\frac{1}{p}}} < \infty$, and hence for $L=M_p \frac{N}{2C^{\frac{1}{p}}}>0$
\[
\bigg\vert \tilde{E}^{(\epsilon)}[\mathcal{O}]-\tilde{E}[\mathcal{O}] \bigg\vert \leq \sum_{l \in A(\mathcal{O})} \varphi_{\epsilon}(\vert O_l \vert) \leq \vert A(\mathcal{O}) \vert L \epsilon^{\frac{1}{p}},
\]
which is precisely \eqref{Eq: Convergence of energy}. \\
Lastly, we see that the compact set $K$ is independent of $\epsilon \in [0,1]$ and satisfies 
\[
\min_{\mathcal{O} \in P_k}\tilde{E}^{(\epsilon)}[\mathcal{O}]=\min_{\mathcal{O} \in K}\tilde{E}^{(\epsilon)}[\mathcal{O}],
\]
which makes the family $\tilde{E}^{(\epsilon)}$ equi-coercive. Together with $\GC_{\epsilon \to 0} \tilde{E}^{(\epsilon)}=\tilde{E}$ from Proposition \ref{Prop: GC of pGMS} and Theorem 1.21 from \cite{Braides}, we get the existence of a convergent subsequence $\hat{\mathcal{O}}^{(\epsilon_j)}$ (not relabeled) with limit $\hat{\mathcal{O}} \in K$, i.e. $d_{P_k}(\hat{\mathcal{O}}^{(\epsilon_j)},\hat{\mathcal{O}}) \to 0$ as $j \to \infty$, and that $\hat{\mathcal{O}}$ is a minimizer of $\tilde{E}$. 
The statistical variables $(\hat{\mu},\hat{\Sigma}) \in S$ are again re-introduced via Lemma \ref{Lem: Partition dependant energies} (with $\epsilon=0$), which yields a minimizer $\hat{x}=(\hat{\mathcal{O}},\hat{\mu},\hat{\Sigma}) \in X$ of $E$. This is the first part of \eqref{Eq: Convergence of the partitions and mean/cov for active labels}. For the second part, note that for any $l \in A(\hat{\mathcal{O}})$, we have $\vert \hat{O}_l^{(\epsilon_j)} \vert \geq \frac{\vert \hat{O}_l \vert }{2}>0$ for $j$ large, and by Lemma \ref{Lem: Partition dependant energies}, the unique statistical variables are given by $\hat{\mu}_l^{(\epsilon_j)}=m_{\hat{O}_l^{(\epsilon_j)}}$ and $\hat{\Sigma}_l^{(\epsilon_j)}=C_{\hat{O}_l^{(\epsilon_j)}}+\epsilon_j I$, which converge by continuity w.r.t. the partitions $\hat{\mathcal{O}}^{(\epsilon_j)}$ as in Lemma \ref{Lem: LSC of the energies}. The conclusion then follows from the same arguments as in Case 1 of the proof of Lemma \ref{Lem: LSC of the energies}. Now that we have proven the existence of minimizers,  the second part of \eqref{Eq: Convergence of energy} follows from the first part of \eqref{Eq: Convergence of energy} and $\vert A(\cdot) \vert \leq k$.
\end{proof}
We would like to close out this section by remarking that for $0 <\epsilon \leq 1$, the existence of minimizers for $E^{(\epsilon)}$ only requires $g \in L^{\infty}((0,1)^2,\mathbb{R}^N)$. We can apply the direct method and note that
\begin{equation}\label{Eq: Lower bound for pGMS}
E^{(\epsilon)}[(\mathcal{O},\mu,\Sigma)] \geq \tilde{E}^{(\epsilon)}[\mathcal{O}] \geq-\frac{N}{2} \ln \frac{1}{2 \pi e \epsilon} + \lambda \sum_{l=1}^k \mathrm{Per}(O_l,(0,1)^2)
\end{equation}
holds for all $(\mathcal{O},\mu,\Sigma) \in X$ and replaces the lower bound from Theorem \ref{Thm: Lower bounds}. The CFC becomes obsolete for $\epsilon>0$ and is replaced by the fact that $C_{O_l}+\epsilon I \succeq \epsilon I \succ 0$ (see Lemma \ref{Lem: Partition dependant energies}). The rest of the argument is similar to the proof of Theorem \ref{Thm: Convergence of minima}.

\subsection{Convergence of alternating minimization}
In order to compute minimizers of $E$ \eqref{Def: GMS} and $E^{(\epsilon)}$ \eqref{Def: pGMS} numerically, we need to use a convex relaxation, in our case Zach's convexification \cite{Zach}. As mentioned before, this is also the main reason why we do not substitute the closed form for $(\mu,\Sigma) \in S$ right away, as it would make the convexification impossible or ill-defined.

For the Zach convexification, we need the unit simplex 
\[
\Delta^k:= \bigg\{z \in \mathbb{R}^k \bigg\vert z_l \geq 0 \text{ for all }1 \leq l \leq k, \sum_{l=1}^k z_l=1   \bigg\}.
\]
We denote by $\Vert f \Vert_{\mathrm{TV}((0,1)^2)}$ the total variation (denoted by $V(f,(0,1)^2)$ in \cite{ambrosio2000functions}) of a function $f \in L^1_{\mathrm{loc}}((0,1)^2)$. We then define the set of soft labels by
\[
U:=\bigg\{ u: (0,1)^2 \to \Delta^k \bigg\vert u \text{ measurable}, \sum_{l=1}^k \Vert u_l \Vert_{\mathrm{TV}((0,1)^2)}< \infty   \bigg\}.
\]
Again, it is notationally convenient to introduce the set of active  labels
\[
A(u):=\{l \in \{1,\dots,k \} \mid \Vert u_l \Vert_{L^1} >0 \}.
\]
The corresponding convex relaxations are
\begin{equation}\label{Eq: Convex relaxations}
\begin{split}
E_{\mathrm{conv}}[u,\mu,\Sigma]&:= \sum_{l=1}^k \bigg( \int_{(0,1)^2}u_l(x)f(g(x) \mid \mu_l,\Sigma_l) \dx + \lambda \Vert u_l \Vert_{\mathrm{TV}((0,1)^2)}   \bigg) \\
E_{\mathrm{conv}}^{(\epsilon)}[u,\mu,\Sigma]&:=E_{\mathrm{conv}}[u,\mu,\Sigma]+ \frac{\epsilon}{2}  \sum_{l=1}^k \Vert u_l \Vert_{L^1} \mathrm{tr}\Sigma_l^{-1} 
\end{split}
\end{equation}
with $(u,\mu,\Sigma) \in U \times (\mathbb{R}^N)^k \times \SPD(N)^k$.
Our update scheme is as follows:

\begin{defn}[Alternating minimization]\label{Def: Alternating min}
Let $g \in \mathcal{A}(p)$ for $p \geq 2$ and $\epsilon \in [0,1]$. Let $(u^{(0)},\mu^{(0)},\Sigma^{(0)}) \in U \times (\mathbb{R}^N)^k \times \SPD(N)^k$. Then we define the \textbf{alternating minimization scheme} by 
\begin{equation}\label{Eq: Update steps}
\begin{aligned}
(\mu^{(n+1)},\Sigma^{(n+1)}) \in &\argmin_{(\mu,\Sigma) \in S} E_{\mathrm{conv}}^{(\epsilon)}[u^{(n)},\mu,\Sigma] \\
u^{(n+1)} \in &\argmin_{u \in U}E_{\mathrm{conv}}^{(\epsilon)}[u,\mu^{(n+1)},\Sigma^{(n+1)}].
\end{aligned}
\end{equation}
\end{defn}
Please note that we suppress the $\epsilon$ index on $(u^{(n)},\mu^{(n)},\Sigma^{(n)})$ for better readability.

We also introduce the \textbf{soft mean} and \textbf{soft covariance} as
\begin{align*}
 m_{u_l}&:=\frac{1}{\Vert u_l \Vert_{L^1}} \int_{(0,1)^2}u_l(x)g(x) \dx \\
C_{u_l}&:=\frac{1}{\Vert u_l \Vert_{L^1}} \int_{(0,1)^2}u_l(x) \bigg( (g(x)-m_{u_l}) \otimes (g(x)-m_{u_l})\bigg) \dx 
\end{align*}
for $ \Vert u_l \Vert_{L^1} > 0$. For numerical purposes, the specific choice for inactive labels $l \notin A(u)$ matters, as we will see in this section's main theorem.

The CFC transfers to the soft covariance matrix:
\begin{lem}[CFC for soft labels]\label{Lem: CFC for soft labels}
Let $g \in \mathcal{A}(p)$ with constant $C>0$ and let $u:(0,1)^2 \to [0,1]$ be measurable with $\Vert u \Vert_{L^1}>0$. Then
\[
C_{u} \succeq C \Vert u \Vert_{L^1}^p I.
\]
\end{lem}
\begin{proof}
Let $y \in S^{N-1}$ be arbitrary. Set $O_t:=\{x \in (0,1)^2 \mid u(x)>t\}$ for $t \in [0,1)$, so that $u(x) = \int_0^1 \mathbf{1}_{O_t}(x) \,\text{d}t$ and $\Vert u \Vert_{L^1} = \int_0^1 \vert O_t \vert \, \text{d}t$ by Fubini. Since the mean minimizes the quadratic deviation, we have for every $t$ with $\vert O_t \vert > 0$,
\[
\int_{O_t} (y^Tg(x)-y^T m_u)^2 \dx \geq \int_{O_t} (y^Tg(x)-y^Tm_{O_t})^2 \dx = \vert O_t \vert  y^T C_{O_t} y \geq C \vert O_t \vert^{p+1}
\]
by the CFC applied to the measurable set $O_t \subset (0,1)^2$. Integrating over $t$ and using Jensen's inequality for the convex function $a \mapsto a^{p+1}$  yields
\begin{align*}
\bigg(\int_{(0,1)^2} u(x) \dx \bigg)  y^T C_u y &= \int_0^1 \int_{O_t} (y^Tg(x)- y^Tm_u )^2 \dx   \text{d}t  \\
& \geq C \int_0^1 \vert O_t \vert^{p+1}  \text{d}t \geq C \bigg( \int_0^1 \vert O_t \vert  \text{d}t \bigg)^{p+1},
\end{align*}
and dividing by $\Vert u \Vert_{L^1} = \int_0^1 \vert O_t \vert  \text{d}t$ yields $y^T C_u y \geq C \Vert u \Vert_{L^1}^p$.
\end{proof}

We record another lemma for the sake of completeness:
\begin{lem}[Closed forms for label dependent energy]\label{Lem: Label dependant energy}
Let $g \in \mathcal{A}(p)$ for $p \geq 2$ and $\epsilon \in [0,1]$. Then, for any $u \in U$, $\min_{(\mu,\Sigma) \in S} E_{\mathrm{conv}}^{(\epsilon)}[u,\mu,\Sigma]$ is uniquely attained by
    \[ 
    \hat{\mu}_l^{(\epsilon)}=m_{u_l} \quad \text{and} \quad \hat{\Sigma}_l^{(\epsilon)}=C_{u_l}+\epsilon I
    \]
    for $l \in A(u)$ and any $(\hat{\mu}_l^{(\epsilon)},\hat{\Sigma}_l^{(\epsilon)}) \in \mathbb{R}^N \times \SPD(N)$ for $l \notin A(u)$. The energy at this minimum is given by
    \[
    E_{\mathrm{conv}}^{(\epsilon)}[u,\hat{\mu}^{(\epsilon)},\hat{\Sigma}^{(\epsilon)}]=\sum_{l \in A(u)} \bigg(
     \frac{\Vert u_l \Vert_{L^1}}{2} \ln \big( (2 \pi e)^N \det (C_{u_l}+ \epsilon I) \big) + \lambda \Vert u_l \Vert_{\mathrm{TV}((0,1)^2)} \bigg).
    \]
\end{lem}
The proof is essentially the same as that of Lemma \ref{Lem: Partition dependant energies} and is therefore omitted. With this, we can show that alternately updating the labels yields a convergent (sub)sequence, at which our update scheme \eqref{Eq: Update steps} comes to a halt:

\begin{thm}[Convergence of alternating minimization]\label{Thm: Convergence of AM}
Let $g \in \mathcal{A}(p)$ for $p \geq 2$ and $\epsilon \in [0,1]$ and $(u^{(0)},\mu^{(0)},\Sigma^{(0)}) \in U \times (\mathbb{R}^N)^k \times \SPD(N)^k$. Let $(u^{(n)},\mu^{(n)},\Sigma^{(n)})$ be the sequence generated by Definition \ref{Def: Alternating min}. This sequence is well-defined and there exist $(u,\mu,\Sigma) \in U \times (\mathbb{R}^N)^k \times \SPD(N)^k $ and a subsequence (not relabeled) such that
\[
u^{(n)} \to u \text{ in } L^1((0,1)^2,\mathbb{R}^k)
\]
and
\[
\vert \mu_l^{(n+1)}-\mu_l \vert +\vert \Sigma_l^{(n+1)}-\Sigma_l \vert_F \to 0 \text{ for all } l \in A(u).
\]
Furthermore, if for $l \notin A(u)$, we set $(\mu_l,\Sigma_l)=(\mu_{l^*},\Sigma_{l^{\ast} })$ for some $l^{\ast} \in A(u)$, we have 
\begin{equation} \label{Eq: Minimality of (u,mu,Sigma)}
 (\mu,\Sigma) \in \argmin_{(\mu',\Sigma') \in S} E_{\mathrm{conv}}^{(\epsilon)}[u,\mu',\Sigma'] \quad \text{and} \quad 
 u \in \argmin_{v \in U}E_{\mathrm{conv}}^{(\epsilon)} [v,\mu,\Sigma].
\end{equation}
\end{thm}
\begin{proof}
Let $g \in \mathcal{A}(p)$ for $p \geq 2$ and $\epsilon \in [0,1]$. Invoking Lemmas \ref{Lem: CFC for soft labels} and \ref{Lem: Label dependant energy}, we get, as in the proof of Theorem \ref{Thm: Lower bounds} and with the same $\psi$ \eqref{Def: psi-function t*ln(t)},
\begin{equation*}
\begin{split}
E_{\mathrm{conv}}^{(\epsilon)}[u,\mu,\Sigma] &\geq \sum_{l \in A(u)} \bigg(
     \frac{\Vert u_l \Vert_{L^1}}{2} \ln \big( (2 \pi e)^N \det (C_{u_l}+ \epsilon I) \big) + \lambda \Vert u_l \Vert_{\mathrm{TV}((0,1)^2)} \bigg) \\
     & \geq \sum_{l=1}^k \bigg(  \frac{N \ln(2\pi e C)}{2} \Vert u_l \Vert_{L^1} + \frac{pN}{2} \psi(\Vert u_l \Vert_{L^1}) + \lambda \Vert u_l \Vert_{\mathrm{TV}((0,1)^2)}  \bigg) \\
     & \geq -K_0 + \sum_{l=1}^k \lambda \Vert u_l \Vert_{\mathrm{TV}((0,1)^2)}
\end{split}
\end{equation*}
with the same $K_0>0$ as in Theorem \ref{Thm: Lower bounds}, in particular independent of $\epsilon>0$. By the direct method, using that bounded sets in $BV((0,1)^2)$ are pre-compact in $L^1((0,1)^2)$, each $u^{(n)}$ is well-defined, since the $\argmin$ is non-empty. The limit lies in $U$ because $\Delta^k$ is closed and $\Vert \cdot \Vert_{\mathrm{TV}((0,1)^2)}$ is lower semi-continuous with respect to $L^1$-convergence \cite{ambrosio2000functions}. In addition, we have $f(g(\cdot) \mid \mu_l,\Sigma_l)+\frac{\epsilon}{2}\tr \Sigma_l^{-1} \in L^{\infty}((0,1)^2)$ for fixed $(\mu,\Sigma) \in S$, which makes the data terms continuous. Moreover, the update steps do not increase the energy, and hence we have
\[
E_{\mathrm{conv}}^{(\epsilon)}[u^{(0)},\mu^{(0)},\Sigma^{(0)}] \geq E_{\mathrm{conv}}^{(\epsilon)}[u^{(n)},\mu^{(n)},\Sigma^{(n)}] \geq -K_0 + \sum_{l=1}^k \lambda \Vert u_l^{(n)} \Vert_{\mathrm{TV}((0,1)^2)}.
\]
As a direct consequence, invoking again the pre-compactness of bounded sets of $BV((0,1)^2)$ in $L^1((0,1)^2)$, we get a subsequence (not relabeled) $u^{(n)} \to u \in U$ in $L^1((0,1)^2,\mathbb{R}^k)$. By the same arguments as in Lemma \ref{Lem: LSC of the energies}, the soft mean and soft covariance converge for $l \in A(u)$. Indeed, for $l \in A(u)$ we have $\Vert u_l^{(n)} \Vert_{L^1} \geq \frac{\Vert u_l \Vert_{L^1}}{2}>0$ for $n$ large, so that $l \in A(u^{(n)})$ and, by Lemma \ref{Lem: Label dependant energy}, $(\mu_l^{(n+1)},\Sigma_l^{(n+1)})=(m_{u_l^{(n)}},C_{u_l^{(n)}}+\epsilon I) \to (m_{u_l},C_{u_l}+\epsilon I)=:(\mu_l,\Sigma_l)$ as in Case 1 of the proof of Lemma \ref{Lem: LSC of the energies}. Note that the sequence index is shifted, since these are the statistical variables generated from $u^{(n)}$.
For the last part, we note that minimality of $(\mu,\Sigma) \in S$ is given by Lemma \ref{Lem: Label dependant energy}, since for inactive labels $ l \notin A(u)$, we can in particular choose $(\mu_l,\Sigma_l)=(\mu_{l^{\ast} },\Sigma_{l^{\ast} })$. This is the first part of \eqref{Eq: Minimality of (u,mu,Sigma)}. For the second part, let $v \in U$ be arbitrary and define the re-arrangement
\[
\tilde{v}_l(x)= \begin{cases}
v_{l^{\ast} }(x)+ \sum_{l' \notin A(u)} v_{l'}(x) & \text{for }l=l^{\ast} \\
v_l(x)   &\text{for }l \in A(u) \setminus \{l^{\ast}  \} \\
0 &\text{for }l \notin A(u).
\end{cases}
\]
We have $\tilde{v} \in U$, since 
\[
\Vert \tilde{v}_{l^{\ast} } \Vert_{\mathrm{TV}((0,1)^2)} \leq \Vert v_{l^{\ast} } \Vert_{\mathrm{TV}((0,1)^2)}+  \sum_{l \notin A(u)}\Vert v_{l} \Vert_{\mathrm{TV}((0,1)^2)}
\]
by the triangle inequality. Hence we have $\sum_{l=1}^k \Vert \tilde{v}_{l} \Vert_{\mathrm{TV}((0,1)^2)}  \leq \sum_{l=1}^k \Vert v_{l} \Vert_{\mathrm{TV}((0,1)^2)}$ and our re-arrangement does not increase the perimeter terms. For the data terms, setting $\phi^{(\epsilon)}(x \mid \mu_l,\Sigma_l):=f(g(x) \mid \mu_l,\Sigma_l)+ \frac{\epsilon}{2}\tr \Sigma_l^{-1}$, we compute
\begin{align*}
&\sum_{l=1}^k \int_{(0,1)^2} \tilde{v}_l(x) \phi^{(\epsilon)}(x \mid \mu_l,\Sigma_l)\dx \\
=&\int_{(0,1)^2} \big( v_{l^{\ast} }(x)+\sum_{l' \notin A(u)}v_{l'}(x) \big) \phi^{(\epsilon)}(x \mid \mu_{l^{\ast} },\Sigma_{l^{\ast} })\dx \\
&\quad+\sum_{l \in A(u) \setminus \{l^{\ast}  \}} \int_{(0,1)^2} v_l(x) \phi^{(\epsilon)}(x \mid \mu_l,\Sigma_l) \dx \\
=&\sum_{l=1}^k \int_{(0,1)^2} v_l(x) \phi^{(\epsilon)}(x \mid \mu_l,\Sigma_l)\dx 
\end{align*}
and hence this re-arrangement leaves the data terms invariant. It remains to compare the energies of $u$ and $\tilde{v}$. Since the update steps do not increase the energy and the energies are bounded below, they converge to some $E_{\infty} \in \mathbb{R}$. By the lower semi-continuity of $\Vert \cdot \Vert_{\mathrm{TV}((0,1)^2)}$, the convergence of the data terms of the labels $l \in A(u)$ and the fact that the data terms of the labels $l \notin A(u)$ are, by Lemma \ref{Lem: Label dependant energy}, of the form $\frac{\Vert u_l^{(n)} \Vert_{L^1}}{2} \ln \big( (2 \pi e)^N \det (C_{u_l^{(n)}}+ \epsilon I) \big) \to 0$ by Case 2 of the proof of Lemma \ref{Lem: LSC of the energies}, we get
\[
E_{\mathrm{conv}}^{(\epsilon)}[u,\mu,\Sigma] \leq \liminf_{n \to \infty}E_{\mathrm{conv}}^{(\epsilon)}[u^{(n)},\mu^{(n+1)},\Sigma^{(n+1)}]=E_{\infty}.
\]
For the active labels of $\tilde{v}$, we have $A(\tilde{v}) \subset A(u)$ and $u^{(n+1)}$  is a minimizer of $E_{\mathrm{conv}}^{(\epsilon)}[\cdot,\mu^{(n+1)},\Sigma^{(n+1)}]$, and therefore we have
\[
E_{\infty} \leq E_{\mathrm{conv}}^{(\epsilon)}[u^{(n+1)},\mu^{(n+1)},\Sigma^{(n+1)}] \leq E_{\mathrm{conv}}^{(\epsilon)}[\tilde{v},\mu^{(n+1)},\Sigma^{(n+1)}] \xrightarrow{n \to \infty} E_{\mathrm{conv}}^{(\epsilon)}[\tilde{v},\mu,\Sigma].
\]

We therefore conclude that
\[
E_{\mathrm{conv}}^{(\epsilon)}[u,\mu,\Sigma] \leq E_{\mathrm{conv}}^{(\epsilon)}[\tilde{v},\mu,\Sigma] \leq E_{\mathrm{conv}}^{(\epsilon)}[v,\mu,\Sigma],
\]
which implies the second part of \eqref{Eq: Minimality of (u,mu,Sigma)}.
\end{proof}

We would like to close out this section by noting that extracting any further properties of the sequence $(u^{(n)},\mu^{(n)},\Sigma^{(n)})$, such as convergence of the full sequence or convergence to a critical point/minimum, is beyond the scope of this paper and would require extensive analysis \cite{ConvergenceCrit}.

\section{Numerical experiments}
We would like to present numerical evidence for our main theorems. As our canonical example, we choose the chessboard, which is given in analytical form by 
\begin{equation}\label{Def: Chessboard}
g(x)= \frac{1}{2} \big( 1-\mathrm{sgn}\big(\sin(8\pi x_1)\sin(8 \pi x_2)\big)\big) \in L^{\infty}((0,1)^2),
\end{equation}
where $\mathrm{sgn}$ denotes the signum function with $\mathrm{sgn}(0)=0$.  We note that the essential range of $g$ is $\{0,1\}$ and that the pushforward measure is given by $\nu=\frac{1}{2}\delta_0 + \frac{1}{2}\delta_1$. Note that on any (subset of a) level set $O_1$ of $g$ with positive measure and finite perimeter, we get $C_{O_1}=0$ and we can make the energy of the GMS \eqref{Def: GMS} arbitrarily low by setting $\Sigma_1=C_{O_1}+ \epsilon I$ and sending $\epsilon \to 0$. In practice, one would of course regularize such an expression by considering $C_{O_1}+\epsilon I$ for a fixed $\epsilon>0$, which is exactly what the pGMS \eqref{Def: pGMS} enforces automatically. By Theorem \ref{Thm: Noise makes images admissible}, we can add (a wavelet-based approximation \cite{fractionalgaussian} of a) fractional Gaussian field $W$ of Hurst index $0.9$  (restricted to $(0,1)^2$) to get the noisy image $g+0.1\cdot W$.

\begin{figure}[H]
    \centering
    \includegraphics[scale=0.33]{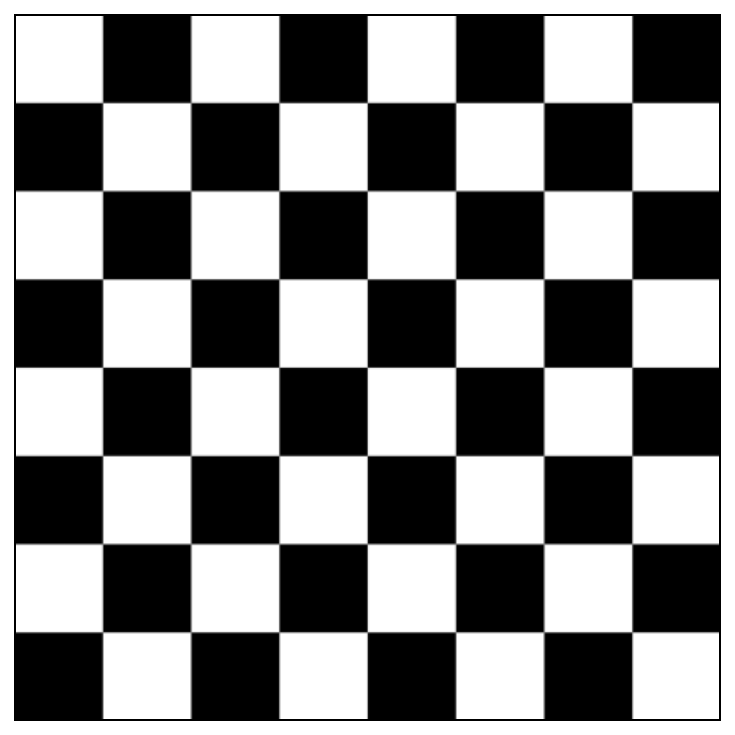}
    \includegraphics[scale=0.33]{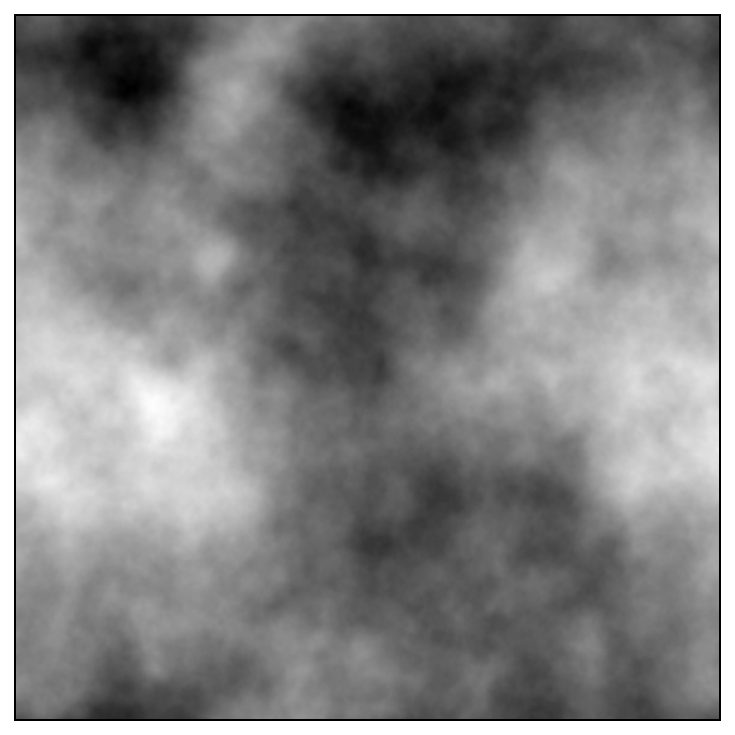}
    \includegraphics[scale=0.33]{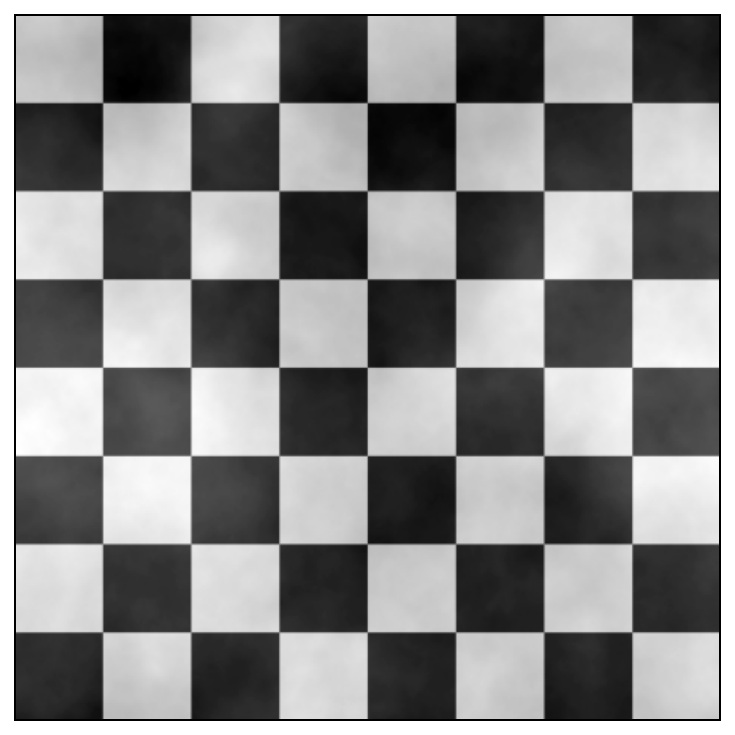}
    \caption{Illustration of the image $g$, the noise $W$ and the noisy image $g+0.1\cdot W$.}
\end{figure}

Theorem \ref{Thm: Noise makes images admissible} applied with $g\equiv0$ tells us that the noise $W$ itself satisfies the CFC. By Theorem \ref{Thm: Equivalence} applied for $p=2$, the CFC is equivalent to our image having a Lipschitz continuous cumulative distribution function $F(t)$. We see that the cumulative distribution of $g$ has a jump type discontinuity at $\{0,1\}$ and hence would violate the RDC and therefore, we have $g \notin \bigcup_{p \geq 2}\mathcal{A}(p)$.
\begin{figure}[H]
    \centering
    \includegraphics[scale=0.33]{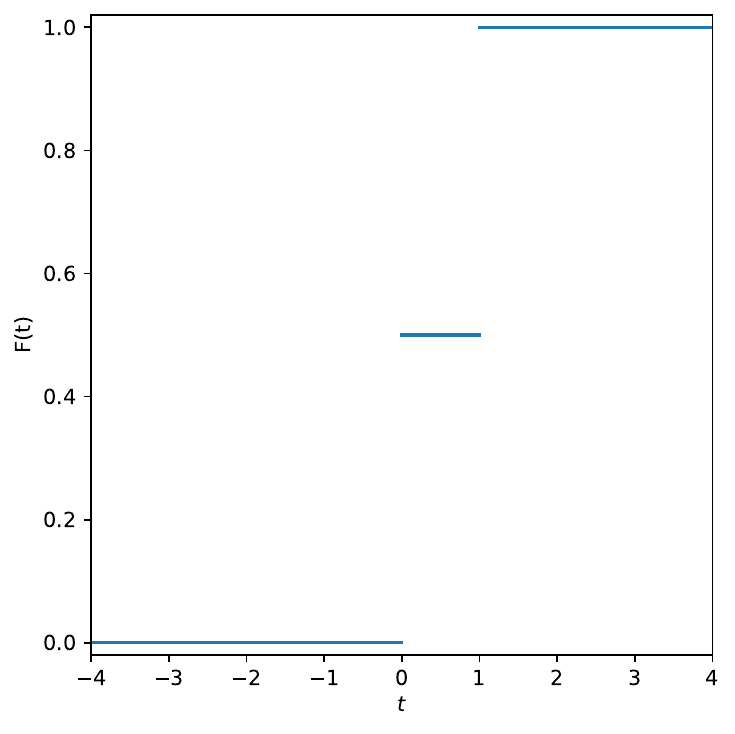}
    \includegraphics[scale=0.33]{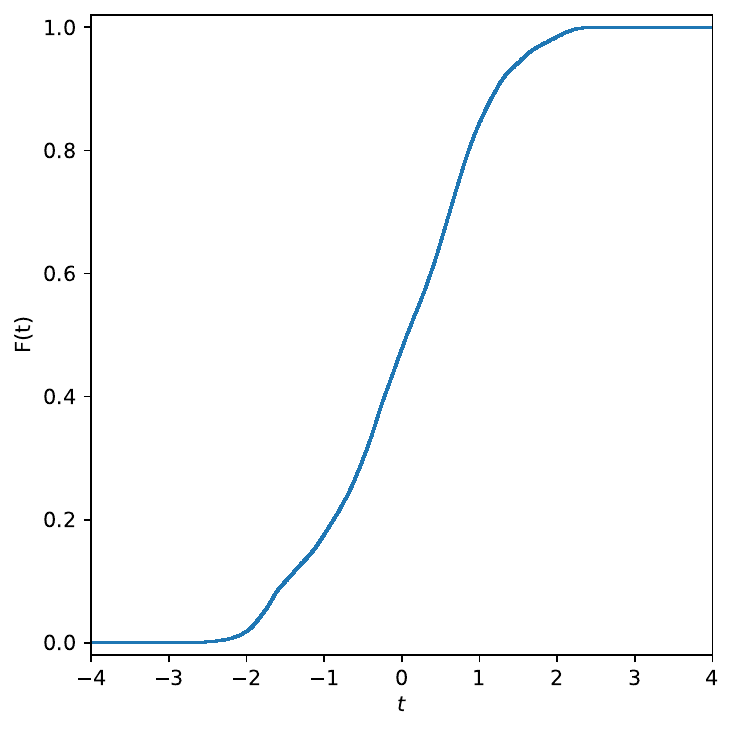}
    \includegraphics[scale=0.33]{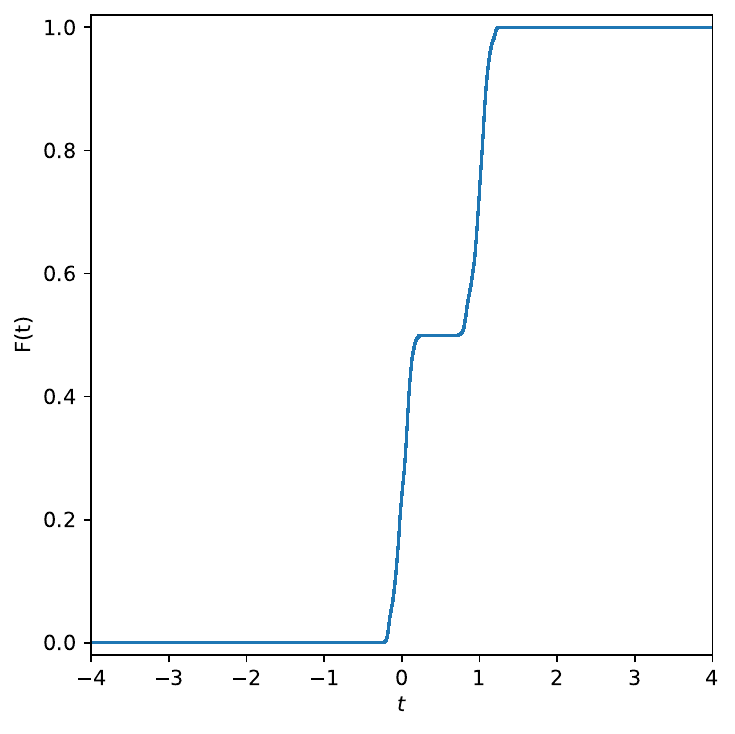}
    \caption{Cumulative distributions of the image $g$, the noise $W$ and the noisy image $g+0.1\cdot W$.}
\end{figure}
We also plot the modulus of continuity $\omega(r)=\sup_{t \in \mathbb{R}} \big(F(t+r)-F(t)\big)$ to see that it also matches our theoretical prediction $\omega(r)\leq C r$ coming from the RDC.
\begin{figure}[H]
    \centering
    \includegraphics[scale=0.33]{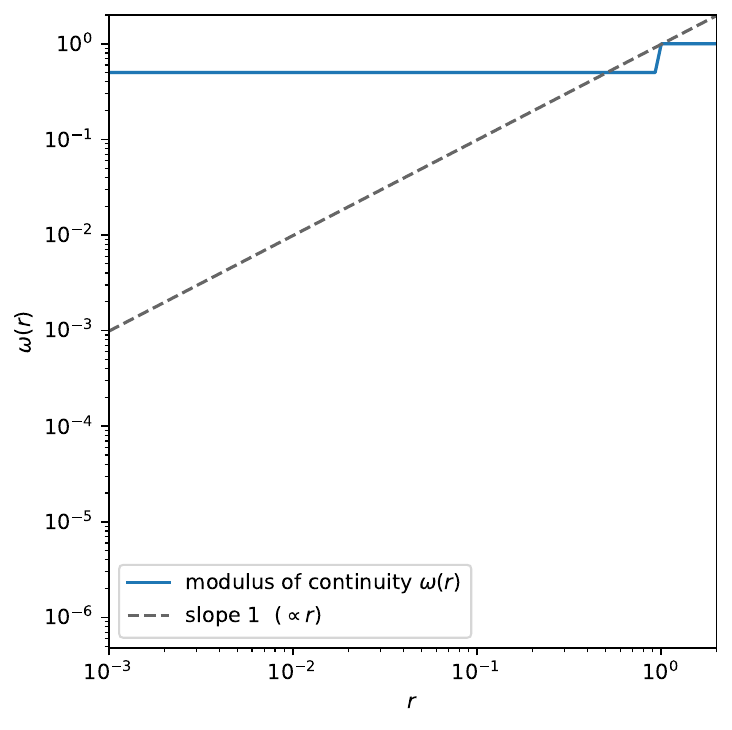}
    \includegraphics[scale=0.33]{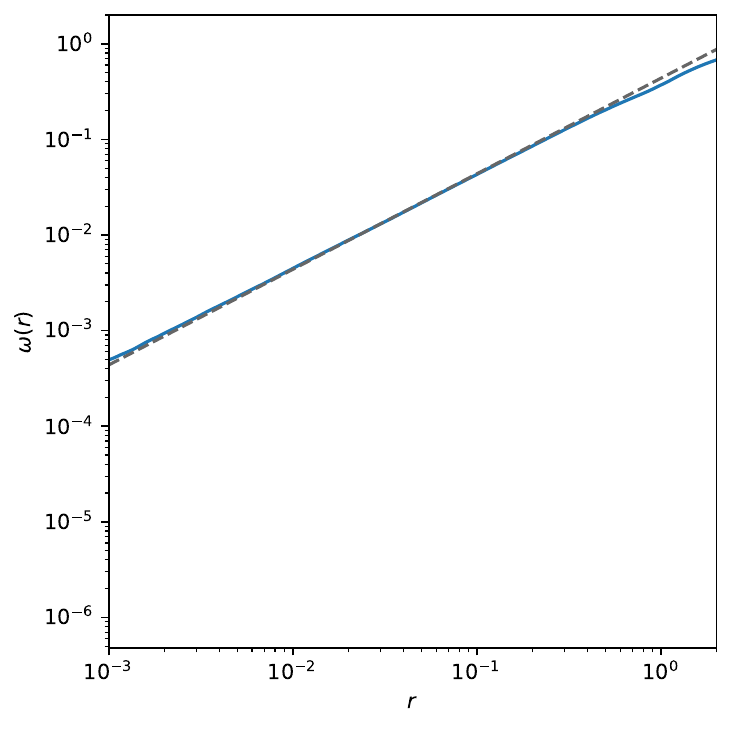}
    \includegraphics[scale=0.33]{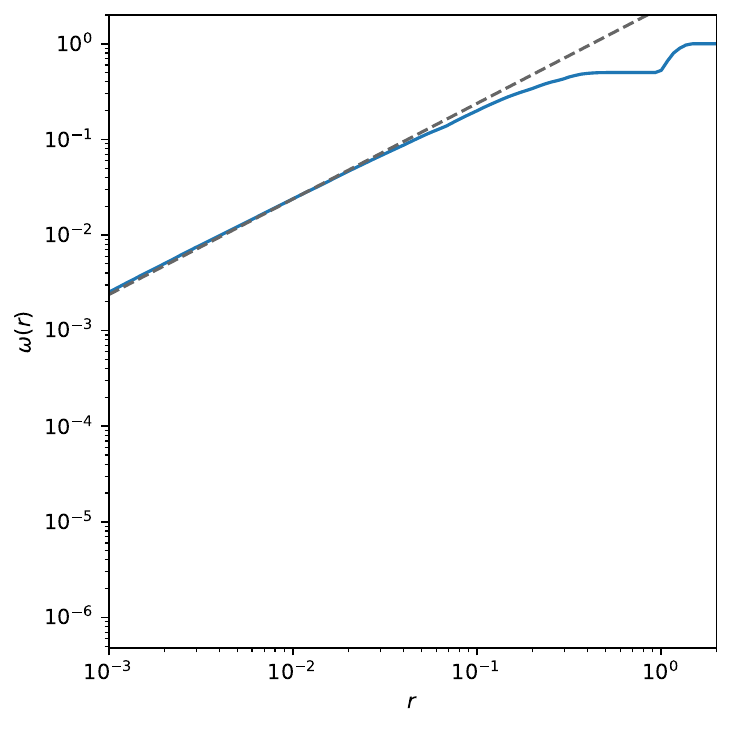}
    \caption{Modulus of continuity of the image $g$, the noise $W$ and the noisy image $g+0.1\cdot W$ plotted on  logarithmic scales.}
\end{figure}

$g+0.1 \cdot W$ is represented as a $1024 \times 1024$ grayscale image and we refer to \cite{Conrad_Berkels_2025_Hyperspectral} for further implementation and discretization details, but we sketch it here for completeness. To find (or rather approximate) a minimizer $(\hat{\mathcal{O}},\hat{\mu},\hat{\Sigma})$ of $E$ numerically, we use the convex relaxation \eqref{Eq: Convex relaxations}. To solve the relaxed problem, we use the alternating minimization (AM) from Definition \ref{Def: Alternating min} with a random initialization and update the labels $u \in U$ with the Primal-Dual Hybrid Gradient (PDHG) algorithm \cite{PDHG} and the statistical variables $(\mu,\Sigma) \in S$ in closed form given by Lemma \ref{Lem: Label dependant energy}. In this case, Theorem \ref{Thm: Convergence of AM} tells us that a (sub)sequence generated by the alternating minimization converges to a point $(\hat{u},\hat{\mu},\hat{\Sigma}) \in U \times (\mathbb{R}^N)^k \times \SPD(N)^k$ minimizing the convexified energy coordinatewise.

We use the following parameters for our numerical segmentation:
\begin{table}[H]
\centering
\caption{Parameters used for the segmentation.}
\label{tab:label}
\begin{tabular}{ll}
\hline
Parameter  & Value \\
\hline
$k$ & $2$ \\
$\lambda$ & $10^{-4}$ \\
AM loops & $4$ \\ 
PDHG steps & $250$ \\
\hline
\end{tabular}
\end{table}

From $\hat{u} \in U$, we recover a partition $\hat{\mathcal{O}} \in P_k$ by setting 
\[
\ell(x)[\hat{u}]:=\min \big(\argmax_{l \in \{1,2\}} \hat{u}_l(x) \big)
\]
and 
\[
\hat{O}_l= \{x \in (0,1)^2 \mid \ell(x)[\hat{u}]=l \} \quad (l=1,2).
\]
The resulting partition $\hat{\mathcal{O}}=(\hat{O}_1,\hat{O}_2)$ is of course an acceptable segmentation for our example, matching exactly the original chessboard $g$.
\begin{figure}[H]
    \centering
    \includegraphics[scale=0.5]{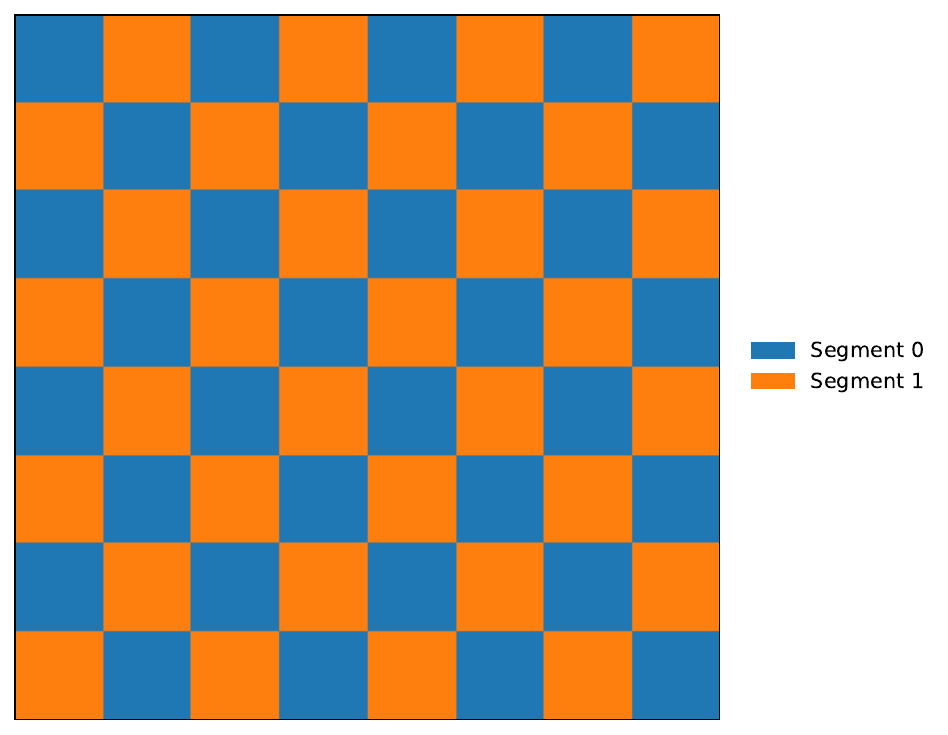}    
    \caption{Segmentation $(\epsilon=0)$ of the noisy image $g + 0.1 \cdot W$.}
\end{figure}
To see that the hypothesis $g \in \mathcal{A}(p)$ cannot be dropped for the existence of a lower bound in Theorem \ref{Thm: Lower bounds}, we examine the energy $\tilde{E}^{(\epsilon)}[\hat{\mathcal{O}}]$ as $\epsilon \to 0$ for the (clean) chessboard $g \in L^{\infty}((0,1)^2) \setminus \bigcup_{p \geq 2} \mathcal{A}(p) $. The pGMS has according to \eqref{Eq: Lower bound for pGMS} an $\epsilon$-dependent lower bound, which goes to $-\infty$ as $\epsilon \to 0$, and it turns out that this lower bound is exact for our test image:
\begin{figure}[H]
    \centering
    \includegraphics[scale=0.75]{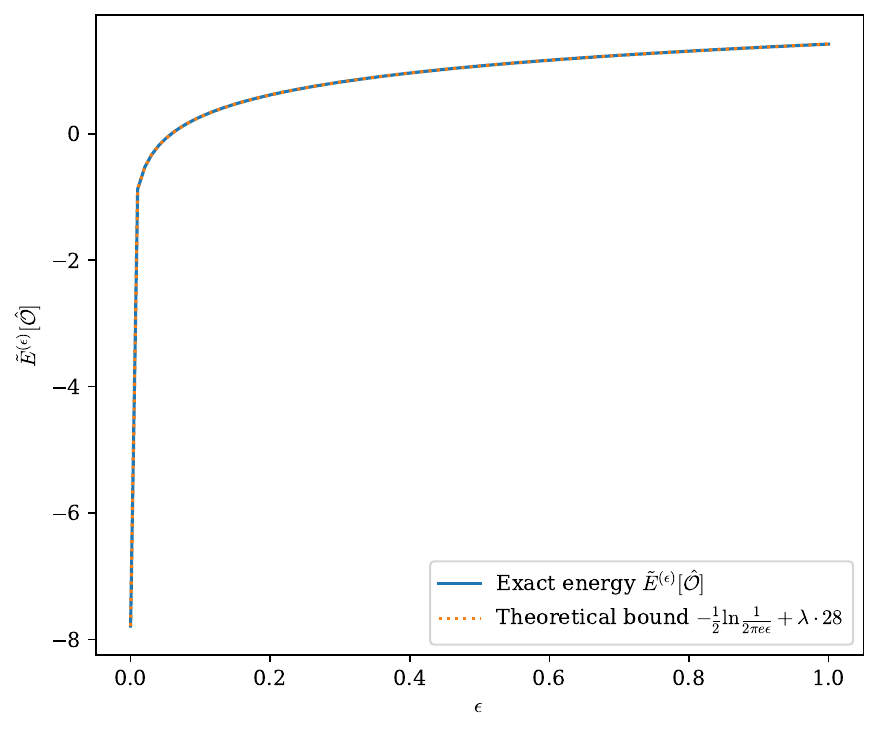}    
    \caption{pGMS energy of the chessboard image $g$ as $\epsilon \to 0$.}
\end{figure}

To test the convergence rate of $ \bigg\vert \tilde{E}^{(\epsilon)}[\hat{\mathcal{O}}]-\tilde{E}[\hat{\mathcal{O}}] \bigg\vert \leq \vert A(\hat{\mathcal{O}}) \vert L \epsilon^{\frac{1}{2}}$  from Theorem \ref{Thm: Convergence of minima} for the noisy chessboard $g+0.1 \cdot W$, we approximate the constant from the CFC by 
\[
C \approx \min_{l \in A(\hat{\mathcal{O}}) } \frac{C_{\hat{O}_l}}{\vert \hat{O}_l \vert^2}
\]
and using $M_2=\sup_{u>0} \frac{\ln (1+u)}{\sqrt{u}} \approx 0.8045$, we get $L \approx 0.8045 \frac{1}{2\sqrt{C}} $. The convergence rate holds uniformly over $\mathcal{O} \in P_k$ and can only be attained if the segments shrink with $\epsilon$: equality in $\varphi_\epsilon(|O_l|) \leq L\epsilon^{\frac{1}{p}}$ requires $\frac{\epsilon}{C|O_l|^p} = u_\ast$, where $u_\ast \approx 3.92$ maximizes $u \mapsto \frac{\ln(1+u)}{\sqrt{u}}$. 
For the fixed partition $\hat{\mathcal{O}}$, the eigenvalues $\kappa_l = C_{\hat{O}_l} \approx 10^{-2}$
are bounded away from zero, so the exact difference $\sum_{l} \frac{|\hat{O}_l|}{2} \ln \big(1 + \frac{\epsilon}{\kappa_l}\big)$ is of order $\epsilon$ compared to our bound of order $\epsilon^{\frac{1}{2}}$ as $\epsilon \to 0$.
\begin{figure}[H]
    \centering
    \includegraphics[scale=0.75]{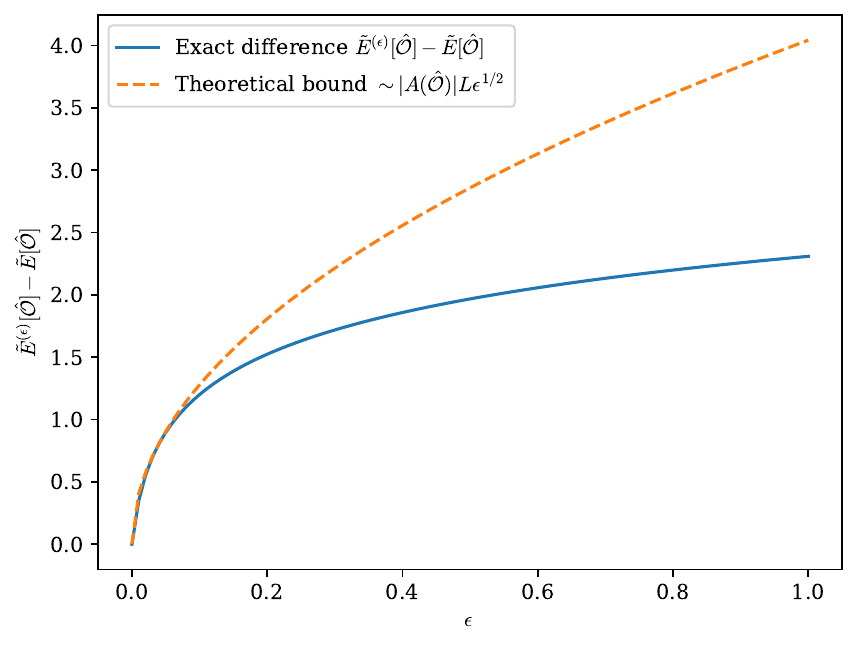}    
    \caption{Convergence of the reduced pGMS  energy to the reduced GMS energy of the image $g+ 0.1 \cdot W$ for the chessboard partition.}
\end{figure}

To perform experiments for RGB/HSI images ($N>1$), one has to study all projections $y^Tg$ for $y \in S^{N-1}$, and they are studied like grayscale images, and one also has to ensure that all numerical experiments, tests and results hold uniformly in $y \in S^{N-1}$. The choice of $N=1$ sidesteps this difficulty. The choice of a synthetic test image was also deliberate for the following reasons. First, the mode of failure is unambiguous: the pushforward measure of $g$ is purely atomic, so $g$ violates the CFC for every $C>0$ and $p>0$, regardless of discretization. A real image may already possess a Hölder continuous cumulative distribution, in which case the experiment of adding noise would demonstrate nothing. Second, we have applied this model to hyperspectral images \cite{Conrad_Berkels_2025_Hyperspectral}, where $N \approx 10$ after dimensionality reduction, but there are often segments with fewer than $N+1$ pixels. Recall that the discrete sample covariance matrix needs at least $N+1$ affinely independent pixels to be strictly positive definite. In these cases, the degeneracy of $C_{O_l}$ is a discretization effect rather than a property of the continuum image; it is the latter that we examine here. Third, the chessboard has no questionable ground truth, with only two natural choices for the number of segments ($k=2$ or $k=64$), so our experiments involve minimal methodological choices and test our theorems rather than the accuracy of the segmentation pipeline.

\section{Conclusion}
We have identified the pushforward measure $\nu$ of an image $g \in L^{\infty}((0,1)^2, \mathbb{R}^N)$ as the only invariant relevant for the existence of covariance floors. Theorem \ref{Thm: Equivalence} shows that, for an image, admitting covariance floors is equivalent to having a uniform Hölder norm bound on the marginal cumulative distributions $F_y$, and the condition is therefore invariant under measure-preserving re-arrangements of the domain $(0,1)^2$.
In particular, spatial regularity of $g$ is irrelevant to the CFC. The condition $p \geq 2$ rules out any scale-independent bound $C_O \succeq CI$ for $C>0$. 
 $g$ satisfying the CFC is generic in practice, as the noisy images satisfy $g + \sigma W \in \mathcal{A}(2)$ and these noisy images are actually dense in $L^{\infty}((0,1)^2,\mathbb{R}^N)$. For images $ g \in \mathcal{A}(p)$, the Gaussian Mumford--Shah model admits a lower bound given by the segment volumes $\vert O_l \vert$ and entropies $H(\mathcal{N}(m_{O_l},C_{O_l}))$. As it turns out, this lower bound combined with standard analysis yields existence of minimizers for the GMS, convergence of minimizers of the pGMS with convergence rates for the energies and convergence of a subsequence generated by the alternating minimization. The continuous setting requires different methods than the discrete setting, as the discrete setting has a prescribed minimum scale and only finitely many possible segments, but also requires each segment to have a minimum number of pixels.
In the continuous setting, the CFC constrains $\nu$ alone and is insensitive to resolution, whereas a discrete segment containing fewer than 
$N+1$ affinely independent pixels has singular sample covariance for reasons unrelated to the value distribution $\nu$. We isolate the mechanism in the continuous setting in our numerical experiments; however, in the hyperspectral image segmentation setting with $N\gg 1$, both the discrete and continuous obstructions are present. 
It remains open how to deal with all images $g \in L^{\infty}((0,1)^2, \mathbb{R}^N)$, in particular piecewise constant images, without falling back to perturbations (such as the pGMS) or restricting the set of possible $\Sigma$.

\printbibliography

\end{document}